\documentclass[10pt]{article}
\pdfoutput=1
\usepackage{fullpage}
\usepackage{natbib}
\usepackage{hyperref}
\usepackage{graphicx}
\usepackage{booktabs}
\usepackage{url}
\usepackage{algorithm}
\usepackage[noend]{algpseudocode}
\usepackage[shortcuts]{extdash}
\usepackage[bm,algpseudocodecmds,subfigurepackage,nocitepackage]{pauli_all}
\usepackage{xspace}
\usepackage{color}
\usepackage{footmisc}
\usepackage{rotating}
\usepackage{topcapt}

\newcommand{\outprod}{\ensuremath{\boxtimes}}
\newcommand{\kprod}{\ensuremath{\otimes}}
\newcommand{\krprod}{\ensuremath{\odot}}
\newcommand{\simi}{\ensuremath{\mathsf{sim}}}

\newcommand{\BCPC}{\ensuremath{\text{BCPC}\xspace}}
\newcommand{\BCPCm}{\ensuremath{\text{BCPC}_{\max}\xspace}}

\newcommand{\Saboteur}{\algname{SaBoTeur}\xspace}
\newcommand{\SaboteurIU}{\algname{SaBoTeur+itUp}\xspace}
\newcommand{\BCPALS}{\algname{BCP\_ALS}\xspace}
\newcommand{\walknmerge}{\algname{Walk'n'Merge}\xspace}
\newcommand{\ParCube}{\algname{ParCube}\xspace}
\newcommand{\ParCuber}{\algname{ParCube}$_{0/1}$\xspace}
\newcommand{\CPAPR}{\algname{CP\_APR}\xspace}
\newcommand{\CPAPRr}{\algname{CP\_APR}$_{0/1}$\xspace}

\newcommand{\ResolverL}{\datasetname{ResolverL}\xspace}
\newcommand{\Resolver}{\datasetname{Resolver}\xspace}
\newcommand{\Enron}{\datasetname{Enron}\xspace}
\newcommand{\TracePort}{\datasetname{TracePort}\xspace}
\newcommand{\Facebook}{\datasetname{Facebook}\xspace}
\newcommand{\Delicious}{\datasetname{Delicious}\xspace}
\newcommand{\Lastfm}{\datasetname{Last.FM}\xspace}
\newcommand{\Mag}{\datasetname{MAG}\xspace}
\newcommand{\Movielens}{\datasetname{MovieLens}\xspace}
\newcommand{\Yago}{\datasetname{YAGO}\xspace}

\newlength{\figwidth}
\newlength{\subfigwidth}
\newlength{\figsep}
\newlength{\subfigspace}
\graphicspath{{figs/}}

\begin{document}
%\mainmatter
%\title{Clustering Boolean Tensors}
%\author{Saskia Metzler and Pauli Miettinen}
%\titlerunning{Clustering Boolean Tensors}
%\authorrunning{Metzler and Miettinen}
%\institute{Max Planck Institut f\"ur Informatik\\ Saarbr\"ucken, Germany\\ \email{\{saskia.metzler, pauli.miettinen\}@mpi-inf.mpg.de}}

%
% --- Author Metadata here ---
%\conferenceinfo{KDD}{'14 New York, NY USA}
%\CopyrightYear{2007} % Allows default copyright year (20XX) to be over-ridden - IF NEED BE.
%\crdata{0-12345-67-8/90/01}  % Allows default copyright data (0-89791-88-6/97/05) to be over-ridden - IF NEED BE.
% --- End of Author Metadata ---

\title{Clustering Boolean Tensors}

%\numberofauthors{2} %  in this sample file, there are a *total*
\author{
Saskia Metzler\\
Max-Planck-Institut f\"ur Informatik\\
Saarbr\"ucken, Germany\\
\texttt{smetzler@mpi-inf.mpg.de}
 \and 
Pauli Miettinen \\
Max-Planck-Institut f\"ur Informatik\\
Saarbr\"ucken, Germany\\
\texttt{pmiettin@mpi-inf.mpg.de}
% \alignauthor
% Saskia Metzler\\
%        \affaddr{Max-Planck-Institut f\"ur Informatik}\\
%        \affaddr{Saarbr\"ucken, Germany}\\
%        \email{saskia.metzler@mpi-inf.mpg.de}
% \alignauthor
% Pauli Miettinen \\
% \affaddr{Max-Planck-Institut f\"ur Informatik}\\
% \affaddr{Saarbr\"ucken, Germany}\\
% \email{pauli.miettinen@mpi-inf.mpg.de}
}

\maketitle
\begin{abstract}
Tensor factorizations are computationally hard problems, and in particular, are often significantly harder than their matrix counterparts. In case of Boolean tensor factorizations -- where the input tensor and all the factors are required to be binary and we use Boolean algebra -- much of that hardness comes from the possibility of overlapping components. Yet, in many applications we are perfectly happy to partition at least one of the modes. In this paper we investigate what consequences does this partitioning have on the computational complexity of the Boolean tensor factorizations and present a new algorithm for the resulting clustering problem. This algorithm can alternatively be seen as a particularly regularized clustering algorithm that can handle extremely high-dimensional observations. We analyse our algorithms with the goal of maximizing the similarity and argue that this is more meaningful than minimizing the dissimilarity. As a by-product we obtain a PTAS and an efficient 0.828-approximation algorithm for rank-1 binary factorizations. Our algorithm for Boolean tensor clustering achieves high scalability, high similarity, and good generalization to unseen data with both synthetic and real-world data sets.
\end{abstract}

\section{Introduction}
\label{sec:intro}
Tensors, multi-way generalizations of matrices, are becoming more common in data mining. Binary 3-way tensors, for example, can be used to record ternary relations (e.g.\ source IP---target IP---target port in network traffic analysis), a set of different relations between the same entities (e.g.\ subject---relation---object data for information extraction), or  different instances of the same relation between the same entities (e.g.\ the adjacency matrices of a dynamic graph over time, such as who sent email to whom in which month).  Given such data, a typical goal in data mining is to find some structure and regularities from it. To that end, tensor decomposition methods are often employed, and if the data is binary, it is natural to restrict also the factors to be binary. This gives rise to the various Boolean tensor decompositions that have recently been studied \citep[e.g.][]{miettinen11boolean,erdos13walknmerge,cerf13closed}. 

Finding good Boolean tensor decompositions is, however, computationally hard. Broadly speaking, this hardness stems from the complexity caused by the overlapping factors. But when the existing Boolean tensor decomposition algorithms are applied to real-world data sets, they often return factors that are non-overlapping in at least one mode. Typically the mode where these non-overlapping factors appear is unsurprising given the data: in the above examples, it would be target ports in network data, relations in information extraction data, and the time (month) in the email data -- in all of these cases, this mode has somewhat different behavior compared to the other two. 

These observations lead to the question: what if we restrict one mode to non-overlapping factors? This removes one of the main sources of the computational complexity -- the overlapping factors -- from one mode, so we would expect the problem to admit better approximability results. We would also expect the resulting algorithms to be simpler. Would the computationally simpler problem transfer to better results? In this paper we study all these questions and, perhaps surprisingly, give an affirmative yes answer to all of them.

Normally, the quality of a clustering or tensor decomposition is measured using the distance of the original data to its representation. In Section~\ref{sec:simil-vs.-diss} we argue, however, that for many data mining problems (especially with binary data), measuring the quality using the similarity instead of distance leads to a more meaningful analysis. Motivated by this, the goal of our algorithms is to maximize the similarity (Section~\ref{sec:algorithm}) instead of minimizing the dissimilarity. As a by-product of the analysis of our algorithms, we also develop and analyse algorithms for maximum-similarity binary rank-1 decompositions in Section~\ref{sec:binary-rank-1}. We show that the problem admits a polynomial-time approximation scheme (PTAS) and present a scalable algorithm that achieves a $2(\sqrt{2} - 1) \approx 0.828$ approximation ratio. The experimental evaluation, in Section~\ref{sec:experiments}, shows that our algorithm achieves comparable representations of the input tensor when compared to methods that do (Boolean or normal) tensor CP decompositions -- even when our algorithm is solving a more restricted problem -- while being generally faster and admitting good generalization to unseen data.

\section{Preliminaries}
\label{sec:preliminaries}

Throughout this paper, we indicate vectors as bold lower\-/case letters ($\vec{v}$), matrices as bold upper\-/case letters ($\matr{M}$), and tensors as bold\ upper\-/case calligraphic letters ($\tens{T}$). 
Element $(i,j,k)$ of a 3-way tensor $\tens{X}$ is denoted as $x_{ijk}$. A colon in a subscript denotes taking that mode entirely; for example, $\matr{X}_{::k}$ is the $k$th \emph{frontal slice} of $\tens{X}$ ($\matr{X}_k$ in short).

A tensor can be \emph{unfolded} into a matrix by arranging its fibers (i.e. its columns, rows, or tubes in case of a 3-way tensor) as columns of a matrix. For a \emph{mode-$n$ matricization}, mode-$n$ fibers are used as the columns and the result is denoted by $\matr{X}_{(n)}$.

The outer product of vectors is denoted by $\outprod$. For vectors $\vec{a}$, $\vec{b}$, and $\vec{c}$  of length $n$, $m$, and $l$, $\tens{X}=\vec{a}\outprod\vec{b}\outprod\vec{c}$ is an \byby{n}{m}{l} tensor with $x_{ijk} = a_ib_jc_k$.

For a tensor $\tens{X}$, $\abs{\tens{X}}$ denotes its number of non-zero elements. The Frobenius norm of a 3-way tensor $\tens{X}$ is $\norm{\tens{X}}=\sqrt{\sum_{i,j,k}x_{ijk}^2}$. If $\tens{X}$ is binary, $\abs{\tens{X}} = \norm{\tens{X}}^2$. The \emph{similarity} between two \byby{n}{m}{l} binary tensors $\tens{X}$ and $\tens{Y}$ is defined as $\simi(\tens{X}, \tens{Y}) = nml - \abs{\tens{X} - \tens{Y}}$.

The \emph{Boolean tensor sum} of binary tensors $\tens{X}$ and $\tens{Y}$ is defined as $(\tens{X}\lor\tens{Y})_{ijk} = x_{ijk}\lor y_{ijk}$.
For binary matrices $\matr{X}$ and $\matr{Y}$ where $\matr{X}$ has $r$ columns and $\matr{Y}$ has $r$ rows their \emph{Boolean matrix product}, $\matr{X}\bprod\matr{Y}$, is defined as $(\matr{X}\bprod\matr{Y})_{ij} = \bigvee_{k=1}^r x_{ik}y_{kj}$. The \emph{Boolean matrix rank} of a binary matrix $\matr{A}$ is the least $r$ such that there exists a pair of binary matrices $(\matr{X}, \matr{Y})$ of inner dimension $r$ with $\matr{A}=\matr{X}\bprod\matr{Y}$.

A binary matrix $\matr{X}$ is a \emph{cluster assignment matrix} if each row of $\matr{X}$ has exactly one non-zero element. In that case the Boolean matrix product corresponds to the regular matrix product, 
\begin{equation}
\label{eq:clustprod}
\matr{X}\bprod\matr{Y} = \matr{X}\matr{Y}\; .
\end{equation}

We can now define the Boolean tensor rank and two decompositions: Boolean CP and Tucker3.

\begin{definition}
  \label{def:Brank}
  The \emph{Boolean tensor rank} of a 3-way binary tensor $\tens{X}$, $\rankB(\tens{X})$, is the least integer $r$ such that there exist $r$ triplets of binary vectors ($\vec{a}_i$, $\vec{b}_i$, $\vec{c}_i$) with $\tens{X} = \bigvee_{i=1}^r \vec{a}_i \outprod \vec{b}_i \outprod \vec{c}_i$\,. 
\end{definition}

The low-rank Boolean tensor CP decomposition is defined as follows.
\begin{problem}[Boolean CP]
 \label{def:BCP}
%  \hspace{0.1em}\newline
  Given an \byby{n}{m}{l} binary tensor $\tens{X}$ and an integer $r$, find binary matrices $\matr{A}$ (\by{n}{r}), $\matr{B}$ (\by{m}{r}), and $\matr{C}$ (\by{l}{r}) such that they minimize $\abs{\tens{X}-\bigvee_{i=1}^r\vec{a}_i\outprod\vec{b}_i\outprod\vec{c}_i}$.
\end{problem}

The standard (non-Boolean) tensor rank and CP decomposition are defined analogously~\citep{kolda09tensor}.
Both finding the least error Boolean CP decomposition and deciding the Boolean tensor rank are NP-hard~\citep{miettinen11boolean}.  Following \citet{kolda09tensor}, we use $[[\matr{A},\matr{B},\matr{C}]]$ to denote the normal 3-way CP decomposition and $[[\matr{A},\matr{B},\matr{C}]]_B$  for the Boolean CP decomposition.

Let $\matr{X}$ be \by{n_1}{m_1} and $\matr{Y}$ be \by{n_2}{m_2} matrix. Their \emph{Kronecker (matrix) product}, $\matr{X}\kprod\matr{Y}$, is the \by{n_1n_2}{m_1m_2} matrix defined by
\[
\matr{X}\kprod\matr{Y}=
\left(\begin{matrix}
  x_{11}\matr{Y} & x_{12}\matr{Y} & \cdots & x_{1m_1}\matr{Y}\\
  x_{21}\matr{Y} & x_{22}\matr{Y} & \cdots & x_{2m_1}\matr{Y}\\
  \vdots            &    \vdots        & \ddots & \vdots              \\
  x_{n_11}\matr{Y} & x_{n_12}\matr{Y} & \cdots & x_{n_1m_1}\matr{Y} 
\end{matrix}\right).
\]

The \emph{Khatri--Rao (matrix) product} of $\matr{X}$ and $\matr{Y}$ is defined as ``column-wise Kronecker''. That is, $\matr{X}$ and $\matr{Y}$ must have the same number of columns ($m_1=m_2=m$), and their Khatri--Rao product $\matr{X}\krprod\matr{Y}$ is the \by{n_1n_2}{m} matrix defined as
\begin{equation}
\label{eq:khatri-rao}
\matr{X}\krprod\matr{Y}=
\begin{pmatrix}
  \vec{x}_1\kprod\vec{y}_1, \vec{x}_2\kprod\vec{y}_2, \ldots, \vec{x}_m\kprod\vec{y}_m
\end{pmatrix}\; .
\end{equation}
Notice that if $\matr{X}$ and $\matr{Y}$ are binary, so are $\matr{X}\kprod\matr{Y}$ and $\matr{X}\krprod\matr{Y}$.

We can write the Boolean CP as matrices using unfolding and matrix products:
\begin{equation}
  \begin{aligned}
  \label{eq:BCP:unfold}
   \matr{X}_{(1)} &= \matr{A}\bprod(\matr{C}\krprod\matr{B})^T \\
    \matr{X}_{(2)} &= \matr{B}\bprod(\matr{C}\krprod\matr{A})^T \\
    \matr{X}_{(3)} &= \matr{C}\bprod(\matr{B}\krprod\matr{A})^T \, .
  \end{aligned}
\end{equation}

The Boolean Tucker3 decomposition can be seen as a generalization of the Boolean CP decomposition:
\begin{problem}[Boolean Tucker3]
  \label{def:BTucker}
  Given an $n$-by-$m$-by-$l$ binary tensor $\tens{X}$ and three integers $r_1$, $r_2$, and $r_3$, find a binary  \byby{r_1}{r_2}{r_3} core tensor $\tens{G}$ and binary factor matrices $\matr{A}$ (\by{n}{r_1}), $\matr{B}$ (\by{m}{r_2}), and $\matr{C}$ (\by{l}{r_3}) such that they minimize
  \begin{equation}
    \label{eq:BTucker}
    \abs*[Big]{\tens{X} - \bigvee_{\alpha=1}^{r_1}\bigvee_{\beta=1}^{r_2}\bigvee_{\gamma=1}^{r_3} g_{\alpha\beta\gamma}\,\vec{a}_{\alpha}\outprod\vec{b}_{\beta}\outprod\vec{c}_{\gamma}}\, .
  \end{equation}
\end{problem}

We use $[[\tens{G}; \matr{A}, \matr{B}, \matr{C}]]_B$ as the shorthand notation for the Boolean Tucker3 decomposition.

%%% Local Variables: 
%%% mode: latex
%%% TeX-master: "main"
%%% End: 

\section{Problem Definitions}
\label{sec:problemDef}

We consider the variation of \emph{tensor clustering} where the idea is to cluster one mode of a tensor and potentially reduce the dimensionality of the other modes. Another common approach is to do the co-clustering equivalent, that is, to cluster each mode of the tensor simultaneously. The former is the approach taken, for example, by \citet{huang08simultaneous}, while the latter appears for instance in \citet{jegelka09approximation}. Unlike either of these methods, however, we concentrate on binary data endowed with the Boolean algebra. 

\subsection{General Problem}
\label{sec:general-problem}

Assuming a 3-way tensors and that we do the clustering in the last mode, we can express the general \emph{Boolean tensor clustering} (BTC) problem as follows:

\begin{problem}[BTC]
  \label{prob:btc}
  Given a binary \byby{n}{m}{l} tensor \tens{X} and integers $k_1$, $k_2$, and $k_3$, find a \byby{k_1}{k_2}{k_3} binary tensor $\tens{G}$ and matrices  $\matr{A}\in\{0,1\}^{n\times k_1}$, $\matr{B}\in\{0,1\}^{m\times k_2}$, and $\matr{C}\in\{0,1\}^{l\times k_3}$ such that $\matr{C}$ is a cluster assignment matrix  and that $[[\tens{G}; \matr{A}, \matr{B}, \matr{C}]]_B$ is a good Boolean Tucker decomposition of $\tens{X}$.
\end{problem}

If we let $k_1 = n$ and $k_2 = m$, we obtain essentially a traditional (binary) clustering problem: Given a set of $l$ binary matrices $\{\matr{X}_{1}, \matr{X}_{2}, \ldots, \matr{X}_{l}\}$ (the frontal slices of $\tens{X}$), cluster these matrices into $k_3$ clusters, each represented by an \by{n}{m} binary matrix $\matr{G}_i$ (the frontal slices of the core tensor $\tens{G}$); the factor matrices $\matr{A}$ and $\matr{B}$ can be left as identity matrices. We call this problem \emph{Unconstrained BTC}. Writing each of the involved matrices as an $nm$-dimensional vector, we obtain the following problem, called the \emph{Discrete Basis Partitioning problem} (DBPP) in~\citep{miettinen08discrete}:
  \begin{problem}[{\citealt[DBPP,][]{miettinen08discrete}}]
    \label{prob:dbpp}
    Given an \by{l}{nm} binary matrix $\matr{X}$ and a positive integer $k$, find matrices $\matr{C}\in\{0,1\}^{l\times k}$ and $\matr{G}\in\{0,1\}^{k\times nm}$ such that $\matr{C}$ is a cluster assignment matrix and $\matr{C}$ and $\matr{G}$ minimize $\norm{\matr{X} - \matr{C}\matr{G}}$.
  \end{problem}

\citet{miettinen08discrete} show that DBPP is \NP-hard and give a $(10+\varepsilon)$-approximation algorithm that runs in polynomial time w.r.t. $n$, $m$, $l$, and $k$, while \citet{jiang14pattern} gives a $2$-approximation algorithm that runs in time $O(nml^k)$.

\subsection{Boolean CP Clustering}
\label{sec:bool-cp-clust}

Constraining $k_1$, $k_2$, and $\tens{G}$ yields to somewhat different looking problems, though. In what can be seen as the other extreme, we can restrict $k_1 = k_2 = k_3 = k$ and $\tens{G}$ (which now is \byby{k}{k}{k}) to hyperdiagonal to obtain the \emph{Boolean CP clustering} (BCPC) problem:

\begin{problem}[BCPC]
  \label{prob:bcpc}
  Given a binary \byby{n}{m}{l} tensor \tens{X} and an integer $k$, find matrices $\matr{A}\in\{0,1\}^{n\times k}$, $\matr{B}\in\{0,1\}^{m\times k}$, and $\matr{C}\in\{0,1\}^{l\times k}$ such that $\matr{C}$ is a cluster assignment matrix  and that $[[\matr{A}, \matr{B}, \matr{C}]]_B$ is a good Boolean CP decomposition of $\tens{X}$.
\end{problem}

What BCPC does is perhaps easiest to understand if we use the unfolding rules~\eqref{eq:BCP:unfold} and write
\begin{equation}
  \label{eq:BCPC:unfold}
  \matr{X}_{(3)} \approx \matr{C}\bprod(\matr{B}\krprod\matr{A})^T\; ,
\end{equation}
where we can see that compared to the general BTC, we have restricted the type of centroids: each centroid must be a row of type $(\vec{b}\kprod\vec{a})^T$. This restriction plays a crucial role in the decomposition, as we shall see shortly. Notice also that using~\eqref{eq:clustprod} we can rewrite~\eqref{eq:BCPC:unfold} without the Boolean product:
\begin{equation}
  \label{eq:BCPC:unfold:alt}
  \matr{X}_{(3)} \approx \matr{C}(\matr{B}\krprod\matr{A})^T\; .
\end{equation}

\subsection{Similarity vs. Dissimilarity}
\label{sec:simil-vs.-diss}

So far we have avoided defining what we mean by ``good'' clustering. Arguably the most common approach is to say that any clustering is good if it minimizes the sum of distances (i.e.\ dissimilarities) between the original elements and their cluster's representative (or centroid). An alternative is to maximize the similarity between the elements and their representatives: if the data and the centroids are binary, this problem is known as \emph{Hypercube segmentation} \citep{kleinberg04segmentation}. While maximizing similarity is obviously a dual of minimizing the dissimilarity -- in the sense that an optimal solution to one is also an optimal solution to the other -- the two behave differently when we aim at approximating the optimal solution. For example, for a fixed $k$, the best known approximation algorithm to DBPP is the aforementioned 2-approximation algorithm \citep{jiang14pattern}, while \citet{alon99two} gave a PTAS for the Hypercube segmentation problem. 

The differences between the approximability, and in particular the reasons behind those differences, are important for data miners. Namely, it is not uncommon that our ability to minimize the dissimilarities is bounded by a family of problems where the optimal solution obtains extremely small dissimilarities, while the best polynomial-time algorithm has to do with slightly larger errors. These larger errors, however, might still be small enough relative to the data size that we can ignore them. This is what essentially happens when we maximize the similarities.

For example, consider an instance of DBPP where the optimal solution causes ten errors. The best-known algorithm can guarantee to not cause more than twenty errors; assuming the data size is, say, \by{1000}{500}, making mistakes on twenty elements can still be considered excellent. On the other hand, should the optimal solution make an error on every fourth element, the best approximation could only guarantee to make errors in at most every second element. When considering the similarities instead of dissimilarities, these two setups look very different: In the first, we get within $0.99998$ of the best solution, while in the other, we are only within $0.66667$ of the optimal. Conversely, if the algorithm wants to be within, say, $0.9$ of the optimal similarity in the latter setup, it must approximate the dissimilarity within a factor of $1.3$. Hence similarity is less sensitive to small errors (relative to the data size) and more sensitive to large errors, than dissimilarity is. In many applications, this is exactly what we want. 

None of this is to say that we should stop considering the error and concentrate only on the similarities. On the contrary: knowing the data size and the error, we can easily compute the similarity for the binary data, and especially when the error is relatively small, it is much easier to compare than the similarity scores. What we do argue, however, is that when analyzing the approximation ratio a data mining algorithm can obtain, similarity can give us a more interesting picture of the actual behavior of the algorithm. Hence, for the rest of the paper, we shall concentrate on the maximum-similarity variants of the problems; most notably, we concentrate on the maximum-similarity BCPC, denoted \BCPCm.

% We consider the variation of \emph{tensor clustering} where the idea is to cluster one mode of a tensor and potentially reduce the dimensionality of the other modes.

% Assuming a 3-way tensor and that we do the clustering in the last mode, we can express the \emph{Boolean CP clustering} (BCPC) problem as follows:

% \begin{definition}[BCPC]
%   \label{prob:bcpc}
%   \hspace{0.1em}\newline
%   Given a binary \byby{n}{m}{l} tensor \tens{X} and an integer $k$, find matrices $\matr{A}\in\{0,1\}^{n\times k}$, $\matr{B}\in\{0,1\}^{m\times k}$, and $\matr{C}\in\{0,1\}^{l\times k}$ such that $\matr{C}$ is a cluster assignment matrix  and that the tuple $(\matr{A}, \matr{B}, \matr{C})$ maximizes $\simi(\tens{X}, [[\matr{A},\matr{B},\matr{C}]]_B)$
% \end{definition}

% To understand what BCPC does, we use the unfolding rules~\eqref{eq:BCP:unfold} and write
% %\begin{equation}
%   %\label{eq:BCPC:unfold}
%   $\matr{X}_{(3)} \approx \matr{C}(\matr{B}\krprod\matr{A})^T\,$,
% %\end{equation}
% where we can see that we have restricted the type of cluster centroids: While in a general clustering problem, we would aim to cluster the frontal slices of $\tens{X}$ into $k$ clusters each represented by an \by{n}{m} matrix, in this setting each cluster representative has to be of type $(\vec{b}\kprod\vec{a})^T$. This restriction on the cluster centroids plays a crucial role in the decomposition, as we shall see shortly. 

%Notice that we can as well write $\matr{C}(\matr{B}\krprod\matr{A})^T)$ ...

%%% Local Variables: 
%%% mode: latex
%%% TeX-master: "main"
%%% End: 

\section{Solving the BCPC\(_{\text{\textbf{max}}}\)}
\label{sec:theory}

Given a tensor $\tens{X}$, for the optimal solution to \BCPCm, we need matrices $\matr{A}$, $\matr{B}$, and $\matr{C}$  that maximize $\simi(\matr{X}_{(3)}, \matr{C}(\matr{B}\krprod\matr{A})^T)$. If we replace $\matr{B}\krprod\matr{A}$ with an arbitrary binary matrix, this would be equal to the Hypercube segmentation problem defined by \citet{kleinberg04segmentation}: Given a set $S$ of $l$ vertices of the $d$-dimensional cube $\{0,1\}^d$, find $k$ vertices $P_1,\ldots,P_k\in \{0,1\}^d$ and a partition of $S$ into $k$ segments to maximize $\sum_{i=1}^k\sum_{c\in S} \simi(P_i, c)$. Hence our algorithms resembles those for Hypercube segmentation, with the added restrictions to the centroid vectors.

\subsection{The Algorithm}
\label{sec:algorithm}
\citet{alon99two} gave an algorithm for the Hypercube segmentation problem that obtains a similarity within $(1-\varepsilon)$ of the optimum. The running time of the algorithm is $\e^{O((k^2/\varepsilon^2)\ln k)}nml$ for \byby{n}{m}{l} data. While technically linear in data size, the first term turns the running time unfeasible even for moderate values of $k$ (the number of clusters) and $\varepsilon$. We therefore base our algorithm on the simpler algorithm by \citet{kleinberg04segmentation} that is based on random sampling. This algorithm obtains an approximation ratio of $0.828-\varepsilon$ with constant probability and running time $O(nmlk(9/\varepsilon)^k\ln(1/\varepsilon))$. While the running time is still exponential in $k$, it is dominated by the number of samples we do: each sample takes time $O(nmlk)$ for $k$ clusters and \byby{n}{m}{l} data. For practical purposes, we can keep the number of samples constant (with the cost of losing approximation guarantees, though).

Our algorithm \Saboteur\ (\textsc{Sa}mpling for \textsc{Bo}olean \textsc{Te}nsor cl\textsc{u}ste\textsc{r}ing), Algorithm~\ref{alg:saboteur}, considers only the unfolded tensor $\matr{X}_{(3)}$. In each iteration, it samples $k$ rows of $\matr{X}_{(3)}$ as the initial, unrestricted centroids. It then turns these unrestricted centroids into the restricted type in line~\ref{alg:saboteur:rank1}, and then assigns each row of $\matr{X}_{(3)}$ to its closest restricted centroid. The sampling is repeated multiple times, and in the end, the factors that gave the highest similarity are returned. 

\begin{algorithm}[tb]
\caption{\Saboteur\ algorithm for the \BCPC}
\label{alg:saboteur}
\small
\begin{algorithmic}[1]
\Input 3-way binary tensor $\tens{X}$, number of clusters $k$, number of samples $r$.
\Output Binary factor matrices $\matr{A}$ and $\matr{B}$, cluster assignment matrix $\matr{C}$.
\Function{\Saboteur}{\tens{X}, k, r}
\Repeat
  \State Sample $k$ rows of $\matr{X}_{(3)}$ into matrix $\matr{Y}$ \label{alg:saboteur:sample}
  \State Find binary matrices $\matr{A}$ and $\matr{B}$ that maximize  $\simi(\matr{Y}, (\matr{B}\krprod\matr{A})^T)$
  \label{alg:saboteur:rank1}
  \State Cluster $\matr{C}$ by assigning each row of $\matr{X}_{(3)}$ to its closest row of  $(\matr{B}\krprod\matr{A})^T$
  \label{alg:saboteur:clusters}
\Until{$r$ resamples are done}
\State \textbf{return} best $\matr{A}$, $\matr{B}$, and $\matr{C}$
\EndFunction
\end{algorithmic}
\end{algorithm}

The algorithm is extremely simple and fast. 
%that, as we shall see in Section~\ref{sec:experiments}, also performs very well. 
In line~\ref{alg:saboteur:sample} the algorithm samples $k$ rows of the data as its initial centroids. \citet{kleinberg04segmentation} proved that among the rows of $\matr{X}_{(3)}$ that are in the same optimal cluster, one is a good approximation of the (unrestricted) centroid of the cluster: 

\begin{lemma}[{\citealt[Lemma 3.1]{kleinberg04segmentation}}]
  \label{lemma:kleinberg}
  Let $\matr{X}$ be an \by{n}{m} binary matrix and let 
  \[
  \vec{y}^* = \argmax_{\vec{y}\in\{0,1\}^m}\sum_{i=1}^n \simi(\vec{x}_i, \vec{y})\; .
  \]
  Then there exist a row $\vec{x}_j$ of $\matr{X}$ such that
  \begin{equation}
    \label{eq:kleinberg}
    \sum_{i=1}^n\simi(\vec{x}_i, \vec{x}_j) \geq (2\sqrt{2} -2)\sum_{i=1}^n\simi(\vec{x}_i, \vec{y}^*)\; .
  \end{equation}
\end{lemma}

That is, if we sample one row from each cluster, we have a high probability of inducing a close-optimal clustering.

\subsection{Binary Rank-1 Matrix Decompositions}
\label{sec:binary-rank-1}

The next part of the \Saboteur\ algorithm is to turn the unrestricted centroids into the restricted form (line~\ref{alg:saboteur:rank1}). We start by showing that this problem is equivalent to finding the maximum-similarity binary rank-1 decomposition of a binary matrix:

\begin{problem}[Binary rank-1 decomposition]
  \label{prob:br1}
  Given an \by{n}{m} binary matrix $\matr{X}$, find an $n$\-/dimensional binary vector $\vec{a}$ and an $m$\-/dimensional binary vector $\vec{b}$ that maximize $\simi(\matr{X}, \vec{a}\outprod\vec{b})$.
\end{problem}

\begin{lemma}
  \label{lemma:centroid_r1}
  Given an \by{k}{nm} binary matrix $\matr{X}$, finding \by{n}{k} and \by{m}{k} binary matrices $\matr{A}$, $\matr{B}$ that maximize $\simi(\matr{X}, (\matr{B}\krprod\matr{A})^T)$ is equivalent to finding the most similar binary rank-1 approximation of each row $\vec{x}$ of $\matr{X}$, where the rows are re-shaped as \by{n}{m} binary matrices.
\end{lemma}

\begin{proof}
  %The proof is a relatively straight-forward re-arrangement of the elements. 
If $\vec{x}_i$ is the $i$th row of $\matr{X}$ and $\vec{z}_i$ is the corresponding row of $(\matr{B}\krprod\matr{A})^T$, then $\simi(\matr{X}, (\matr{B}\krprod\matr{A})^T) = \sum_{i=1}^k\simi(\vec{x}_i, \vec{z}_i)$, and hence we can solve the problem row-by-row. Let $\vec{x} = (x_{1,1}, x_{2,1}, \ldots, x_{n,1}, x_{1,2}, \ldots, x_{n,m})$ be a row of $\matr{X}$. Re-write $\vec{x}$ as an \by{n}{m} matrix $\matr{Y}$, %in column major order, 
\[
\matr{Y} = 
\begin{pmatrix}
  x_{1,1} & x_{1,2} & \cdots & x_{1,m} \\
  x_{2,1} & x_{2,2} & \cdots & x_{2,m} \\
  \vdots & \vdots & \ddots & \vdots \\
  x_{n,1} & x_{n,2} & \cdots & x_{n,m}
\end{pmatrix}\; .
\]

Consider the row of $(\matr{B}\krprod\matr{A})^T$ that corresponds to $\vec{x}$, and notice that it can be written as $(\vec{b}\kprod\vec{a})^T$, where $\vec{a}$ and $\vec{b}$ are the columns of $\matr{A}$ and $\matr{B}$ that correspond to $\vec{x}$. As $(\vec{b}\kprod\vec{a})^T = (b_1\vec{a}^T, b_2\vec{a}^T, \cdots, b_m\vec{a}^T)$, re-writing it similarly as $\vec{x}$ we obtain
\[
%\begin{pmatrix}
%  b_1\vec{a} & b_2\vec{a} & \cdots & b_m\vec{a}
%\end{pmatrix}
%= 
\begin{pmatrix}
  a_1b_1 & a_1b_2 & \cdots & a_1b_m \\
  a_2b_1 & a_2b_2 & \cdots & a_2b_m \\
  \vdots & \vdots & \ddots & \vdots \\
  a_nb_1 & a_nb_2 & \cdots & a_nb_m
\end{pmatrix}
=\vec{a}\vec{b}^T = \vec{a}\outprod \vec{b}\; .
\]
Therefore, $\simi(\vec{x}, (\vec{b}\kprod\vec{a})^T) = \simi(\matr{Y}, \vec{a}\outprod\vec{b})$.
\end{proof}

%Binary rank-1 approximations are studied earlier. For example, Shen et al.~\cite{shen09mining} give a binary integer linear program (ILP) to find the optimum rank-1 approximation together with an LP-relaxation and minimum cut formulation that achieve factor-2 approximation guarantees for the minimum error problem. Another 2-approximation algorithm for minimum error problem is given by Jiang~\cite{jiang14pattern}. Here, however, we will consider the maximum similarity versions. 

We start by showing that the maximum-similarity binary rank-1 decomposition admits a PTAS. To that end, we present a family of randomized algorithms that on expectation attain a similarity within $(1-\varepsilon)$ of the optimum. The family is similar to that of \citeauthor{alon99two}'s~(\citeyear{alon99two}) and can be analysed and de-randomized following the techniques presented by them. The family of algorithms (one for each $\varepsilon$) is presented as Algorithm~\ref{alg:rank1_PTAS}.

\begin{algorithm}[tb]
\caption{Randomized PTAS for maximum-similarity binary rank-1 decompositions}
\label{alg:rank1_PTAS}
\small
\begin{algorithmic}[1]
\Input An \by{n}{m} binary matrix $\matr{X}$.
\Output Binary vectors $\vec{a}$ and $\vec{b}$.
\Function{PTAS$_\varepsilon$}{\matr{X}}
\State Sample $l=\Theta(\varepsilon^{-2})$ rows of $\matr{X}$ u.a.r. with replacements
\ForAll{partitions of the sample into two sets}
  \State Let $\matr{X}'$ contain the rows in the sample
  \For{both sets in the partition}
    \State Let $\vec{a}$ be the incidence vector of the set
    \State Find $\vec{b}^T$ maximizing $\simi(\matr{X}', \vec{a}\vec{b}^T)$ \label{alg:rank1_PTAS:maxb}
    \State Extend $\vec{a}$ to the rows not in the sample maximizing $\simi(\matr{X}, \vec{a}\vec{b}^T)$ \label{alg:rank1_PTAS:maxa}
  \EndFor
\EndFor
\State \textbf{return} best vectors $\vec{a}$ and $\vec{b}$
\EndFunction
\end{algorithmic}
\end{algorithm}

Lines~\ref{alg:rank1_PTAS:maxb} and~\ref{alg:rank1_PTAS:maxa} require us to solve one of the factor vectors when the other is given. To build $\vec{b}$ (line~\ref{alg:rank1_PTAS:maxb}), we simply take the column-wise majority element of those rows where $\vec{a}$ is $1$, and we extend $\vec{a}$ similarly in line~\ref{alg:rank1_PTAS:maxa}.

The analysis of Algorithm~\ref{alg:rank1_PTAS} follows closely the proofs by \citet{alon99two} and is omitted.

The running time of the PTAS algorithm becomes prohibitive even with moderate values of $\varepsilon$, and therefore we will not use it in \Saboteur. Instead, we present a simple, deterministic algorithm that approximates the maximum similarity within $0.828$, Algorithm~\ref{alg:rank1_approx}. It is similar to the algorithm for Hypercube segmentation based on random sampling presented by \citet{kleinberg04segmentation}. The algorithm considers every row of $\matr{X}$ as a potential vector $\vec{b}$ and finds the best $\vec{a}$ given $\vec{b}$.  Using Lemma~\ref{lemma:kleinberg} it is straight forward to show that the algorithm achieves the claimed approximation ratio:

\begin{algorithm}[tb]
\caption{Approximation of maximum-similarity binary rank-1 decompositions}
\label{alg:rank1_approx}
\small
\begin{algorithmic}[1]
\Input An \by{n}{m} binary matrix $\matr{X}$.
\Output Binary vectors $\vec{a}$ and $\vec{b}$.
\Function{$A$}{\matr{X}}
\ForAll{rows $\vec{x}_i$ of $\matr{X}$}
  \State Let $\vec{b} = \vec{x}_i$
  \State Find $\vec{a}$ maximizing $\simi(\matr{X}, \vec{a}\vec{b}^T)$
\EndFor
\State \textbf{return} best vectors $\vec{a}$ and $\vec{b}$
\EndFunction
\end{algorithmic}
\end{algorithm}

\begin{lemma}
  \label{lemma:rank1_approx}
  Algorithm~\ref{alg:rank1_approx} approximates the optimum similarity within $0.828$ in time $O(nm\min\{n,m\})$.
\end{lemma}

\begin{proof}
  To prove the approximation ratio, let $\vec{a}^*(\vec{b}^*)^T$ be the optimum decomposition. Consider the rows in which $\vec{a}^*$ has $1$. Per Lemma~\ref{lemma:kleinberg}, selecting one of these rows, call it $\vec{b}$, gives us $\simi(\matr{X}, \vec{a}^*\vec{b}^T)\geq (2\sqrt{2} -2)\simi(\matr{X}, \vec{a}^*\vec{b}^T)$ (notice that $\vec{a}^*\vec{b}^T$ agrees with the optimal solution in rows where $\vec{a}^*$ is zero). Selecting $\matr{a}$ that maximizes the similarity given $\matr{b}$ can only improve the result, and the claim follows as we try every row of $\matr{X}$. 

If $n< m$, the time complexity follows as for every candidate $\vec{b}$ we have to make one sweep over the matrix. If $m < n$, we can operate on the transpose.  
\end{proof}

\subsection{Iterative Updates}
\label{sec:iterative-updates}
The \Saboteur algorithm bears resemblance to the initialization phase of $k$-means style clustering algorithms. We propose a variation of \Saboteur, \SaboteurIU, where $k$-means style iterative updates are performed on the result of \Saboteur until it does not improve further. In each iteration of the iterative updates phase, the new centroids are first determined by placing $1$s in those positions where the majority of cluster members have a $1$. Next, these new centroids are constrained to rank-1, and finally, the cluster assignment is recomputed. 

 This procedure clearly slows down the algorithm. Also, due to the constraint on the centroids the iterative updates might actually impair the result, unlike in the normal $k$-means algorithm that converges to a local optimum. However, as we shall see in Section~\ref{sec:experiments}, the results in practice improve slightly. 
 
\subsection{Implementation Details}
\label{sec:algo:implementation}
We implemented the \Saboteur algorithm in C.\!\footnote{The code is available from \url{http://www.mpi-inf.mpg.de/~pmiettin/btc/}} 
In the implementation, the tensor $\tens{X}$ is represented as a bitmap. This allows for using fast vectorized instructions such as \texttt{xor} and \texttt{popcnt}. In addition this representation takes exactly one bit per entry (excluding necessary padding), being very memory efficient even compared to sparse tensor representations.

The second ingredient to the fast algorithm we present is parallelization. As every row of $\matr{X}_{(3)}$ is handled independently, the algorithm is embarrassingly parallel. We use the OpenMP \citep{dagum1998openmp} and parallelize along the sampled cluster centroids. This means that the rank-1 decompositions of each centroid  as well as the computation of the similarity in each of the columns of the Khatri-Rao product (line~\ref{alg:saboteur:rank1} of Algorithm~\ref{alg:saboteur}) are computed in parallel.

% \subsection{Discussion}
% \label{sec:other-variants}

% Lemma~\ref{lemma:centroid_r1} gives us yet another way of interpreting \BCPC, namely, in \BCPC\ each centroid must be a binary rank-1 matrix. One could define a more general variant where the centroids are arbitrary-rank binary matrices. Between these two extrema is a problem where the (Boolean) ranks of the centroids are bounded from above by some constant $r< \min\{n,m\}$. For such a problem, however, finding the centroids is even harder than it is now, as it essentially requires us to solve the Boolean matrix factorization problem which is a hard problem even to approximate~\cite{miettinen08discrete}. 

\subsection{Selecting the Number of Clusters}
\label{sec:mdl}

A common problem in data analysis is that many methods require the selection of the \emph{model order} a priori; for example, most low-rank matrix factorization methods require the user to provide the target rank before the algorithm starts. Many clustering methods -- ours included -- assume the number of clusters as a parameter, although there are other methods that do not have this requirement. 

When the user does not know what would be a proper number of clusters for the data, we propose the use of the \emph{minimum description length principle} \citep[MDL,][]{rissanen78modeling} for automatically inferring it. In particular, we adopt the recent work of \citet{miettinen14mdl4bmf} on using MDL for Boolean matrix factorization (BMF) to our setting (we refer the reader to~\citet{miettinen14mdl4bmf} for more information). 

The intuition behind the MDL principle is that the best model for the data (best number of clusters, in our case) is the one that lets us compress the data most. We use the two-part (i.e.\  crude) MDL where the \emph{description length} contains two parts: the description length of the model $M$, $L(M)$, and that of the data $D$ given the model $M$, $L(D\mid M)$. The overall description length is simply $L(M) + L(D\mid M)$. %Intuitively, when model complexity $L(M)$ increases (e.g.\ as more clusters are added), the model fits better to the data and  $L(D\mid M)$ decreases; however, at some point, the model starts over-fitting, and the increase in $L(M)$ overtakes the decrease in $L(D\mid M)$. 
The optimal model (in MDL's sense) is the one where $L(M) + L(D\mid M)$ is minimized. 

In \BCPC, the model consists of the two factor matrices $\matr{A}$ and $\matr{B}$ and the cluster assignment matrix $\matr{C}$. Given these, we can reconstruct the original tensor $\tens{X}$ if we know the locations where $[[\matr{A}, \matr{B}, \matr{C}]]_B$ differs from $\tens{X}$. Therefore, $L(M)$ is the total length of encoding $\matr{A}$, $\matr{B}$, and $\matr{C}$, while $L(D\mid M)$ is the length of encoding the locations of the differences, i.e.\ tensor $\tens{X}\oplus [[\matr{A}, \matr{B}, \matr{C}]]_B$, where $\oplus$ is the element-wise \emph{exclusive-or}. 

Encoding $\matr{A}$ and $\matr{B}$ is similar to BMF, so we can use the DtM encoding of \citet{miettinen14mdl4bmf} for encoding them (and the related dimensions of the data). 
To encode the matrix $\matr{C}$, we only need to store to which cluster each frontal slice is associated to, taking $l\cdot\log_2(k)$ bits. This finalises the computation of $L(M)$.

To compute $L(D\mid M)$, we note that we can re-write 
$
\tens{X}\oplus [[\matr{A}, \matr{B}, \matr{C}]]_B = \matr{X}_{(3)}\oplus \matr{C}\bprod (\matr{B}\bprod\matr{A})^T .
$
The computation of $L(D\mid M)$ now becomes equivalent of computing it for Boolean matrix factorization with factor matrices $\matr{C}$ and $\matr{D}= (\matr{B}\bprod\matr{A})^T$. Hence, we can follow the Typed XOR DtM approach of \citet{miettinen14mdl4bmf} directly.

To sum up, in order to use MDL to select the number of clusters, we have to run \Saboteur with different numbers of clusters, compute the description length for each result, and take the one that yields the smallest description length. 

\subsection{Maximum-Similarity BTC}
\label{sec:algo:btc}

We note here that as the Hypercube segmentation problem is the maximization version of the DBPP (Problem~\ref{prob:dbpp}), the algorithms of \citet{kleinberg04segmentation} and \citet{alon99two} can be used as such for solving the maximum-similarity Unconstrained BTC. 

%%% Local Variables: 
%%% mode: latex
%%% TeX-master: "main"
%%% End: 

\section{Experimental Evaluation}
\label{sec:experiments}

\subsection{Other Methods and Evaluation Criteria}
\label{sec:other-methods}

To the best of our knowledge, this paper presents the first algorithms for Boolean tensor clustering and hence we cannot compare directly to other methods. We decided to compare \Saboteur\ to continuous and Boolean tensor CP decompositions. We did not use other tensor clustering methods as they aim at optimizing significantly different targets (see Section~\ref{sec:relatedWork} for more elaborate discussion). 

We used the \BCPALS~\citep{miettinen11boolean} and \walknmerge~\citep{erdos13walknmerge} algorithms for computing Boolean CP decompositions. \BCPALS\ is based on iteratively updating the factor matrices one at a time (similarly to the classical alternating least squares optimizations), while \walknmerge\ is a recent algorithm for scalable Boolean tensor factorization in sparse binary tensors. We did not use \walknmerge\ on synthetic data as \BCPALS\ is expected to perform better on smaller and denser tensors~\citep{erdos13walknmerge}, but we used it on some larger real-world tensors; \BCPALS, on the other hand, does not scale well to larger tensors and hence we had to omit it from most real-world experiments. 

Of the continuous methods we used \ParCube \citep{papalexakis12parcube}\footnote{\url{http://www.cs.cmu.edu/~epapalex/}} and \CPAPR \citep{chi12tensors} (implementation from the Matlab Tensor Toolbox v2.5\footnote{\url{http://www.sandia.gov/~tgkolda/TensorToolbox/}}). \CPAPR\ is an alternating Poisson regression algorithm that is specifically developed for sparse (counting) data (which can be expected to follow the Poisson distribution), with the goal of returning sparse factors. 

%The aim for sparsity and, to some extend, considering the data as a counting data, make this method suitable for comparison; on the other hand, it aims to minimize the (generalized) K--L divergence, not squared error, and binary data is not Poisson distributed. %\!\footnote{Sampling Poisson distribution can give a binary matrix, but it cannot be forced to give one.}

The other method, \ParCube, uses sampling to find smaller sub-tensors. It then solves the CP decomposition in this sub-tensor, and merges the solutions back into one. We used a non-negative variant of \ParCube\ that expects non-negative data, and returns non-negative factor matrices. 
%\ParCube\ aims to minimize the squared error. 

For synthetic data, we report the relative similarity, that is, the fraction of the elements where the data and the clustering agree. For real-world data, we report the error measured using squared Frobenius norm (i.e.\ the number of disagreements between the data and the clustering when both are binary). Comparing binary methods against continuous ones causes issues, however. Using the squared Frobenius can help the real-valued methods, as it scales all errors less than $1$ down, but at the same time, small errors cumulate unlike with fully binary data. To alleviate this problem, we also rounded the reconstructed tensors from \CPAPR\ and \ParCube\ to binary tensors. We tried different rounding thresholds between $0$ and $1$ and selected the one that gave the lowest (Boolean) reconstruction error. The rounded versions are denoted by \CPAPRr and \ParCuber.

It is worth emphasizing that all of the methods we are comparing against are solving a relaxed version of the BCPC problem. Compared to BCPC, the Boolean CP factorization is not restricted to clustering the third mode while the normal CP factorization lifts even the requirement for binary factor matrices. Hence, a valid solution to BCPC is always also a valid solution to (Boolean and normal) CP factorization (notice, however, that the solutions for Boolean and normal CP factorization are not interchangeable). The methods we compare against should therefore always perform better than (or at least as good as) \Saboteur.

\subsection{Synthetic Experiments}
\label{sec:synth-exper}

To test the \Saboteur\ algorithm in a controlled environment, we created synthetic data sets that measured the algorithm's response to (1) different numbers of clusters, (2) different density of data, (3) different levels of additive noise, and (4) different levels of destructive noise. All tensors were of size \byby{700}{500}{50}. All data sets were created by first creating ground-truth binary factor matrices $\matr{A}$, $\matr{B}$, and $\matr{C}$. The default number of clusters was $7$ and the default density of the tensor $\tens{X} = [[\matr{A}, \matr{B}, \matr{C}]]_B$ was $0.05$. 

Additive and destructive noise were applied to the tensor $\tens{X}$. Additive noise turns zeros into ones while destructive noise turns ones into zeros. The default noise level\footnote{The noise levels are reported w.r.t. number of non-zeros.} for both types was $10\%$, yielding to a noised input tensor $\widetilde{\tens{X}}$. 

We varied each of the four features one at a time keeping the others in their default values, and created $5$ random copies on each parameter combination. The results we report are mean values over these five random copies. In all experiments, the number of clusters (or factors) was set to the true number of clusters used to create the data. The number of re-samples in \Saboteur\ was set to $r=20$ in all experiments. 

For similarity experiments, we compared the obtained reconstructions $[[\tilde{\matr{A}}, \tilde{\matr{B}}, \tilde{\matr{C}}]]_B$ to both the original tensor $\tens{X}$ and the noised tensor $\widetilde{\tens{X}}$.

\subsubsection{Varying the Noise}
\label{sec:varying-noise}
In the first experiment, we studied the effects different noise levels have to the reconstruction accuracy. First, we varied the level of additive noise from $5\%$ to $50\%$ (in steps of $5\%$). The results are in Figure~\ref{fig:synth:additive}, where we can see that \Saboteur (with and without iterative updates) and \CPAPRr all achieve very high similarity with a gentle downwards slope as the level of additive noise increases. \BCPALS\ is slightly worse, but consistent, while the performance of \ParCuber suffers significantly from increased noise levels. In order to test if the good performance of \Saboteur means it is just modeling the noise, we also compared the similarity of the reconstruction to the original, noise-free tensor $\tens{X}$ (Figure~\ref{fig:synth:additive:orig}). This did not change the results in meaningful ways, except that \BCPALS\ improved somewhat.

When moving from additive to destructive noise (Figure~\ref{fig:synth:destructive}), \Saboteur, \SaboteurIU, and \CPAPRr still stay as the best three methods, but as the destructive noise level increases, \BCPALS approaches the three. That behaviour, however, is mostly driven by `modeling the noise', as can be seen in Figure~\ref{fig:synth:destructive:orig}, which shows the similarity to the noise-free data. There, with highest levels of destructive noise, \Saboteur's performance suffers more than \CPAPRr's.

\begin{figure}
  \centering
  \subfigure[Additive noise\label{fig:synth:additive}]{\includegraphics[width=\subfigwidth]{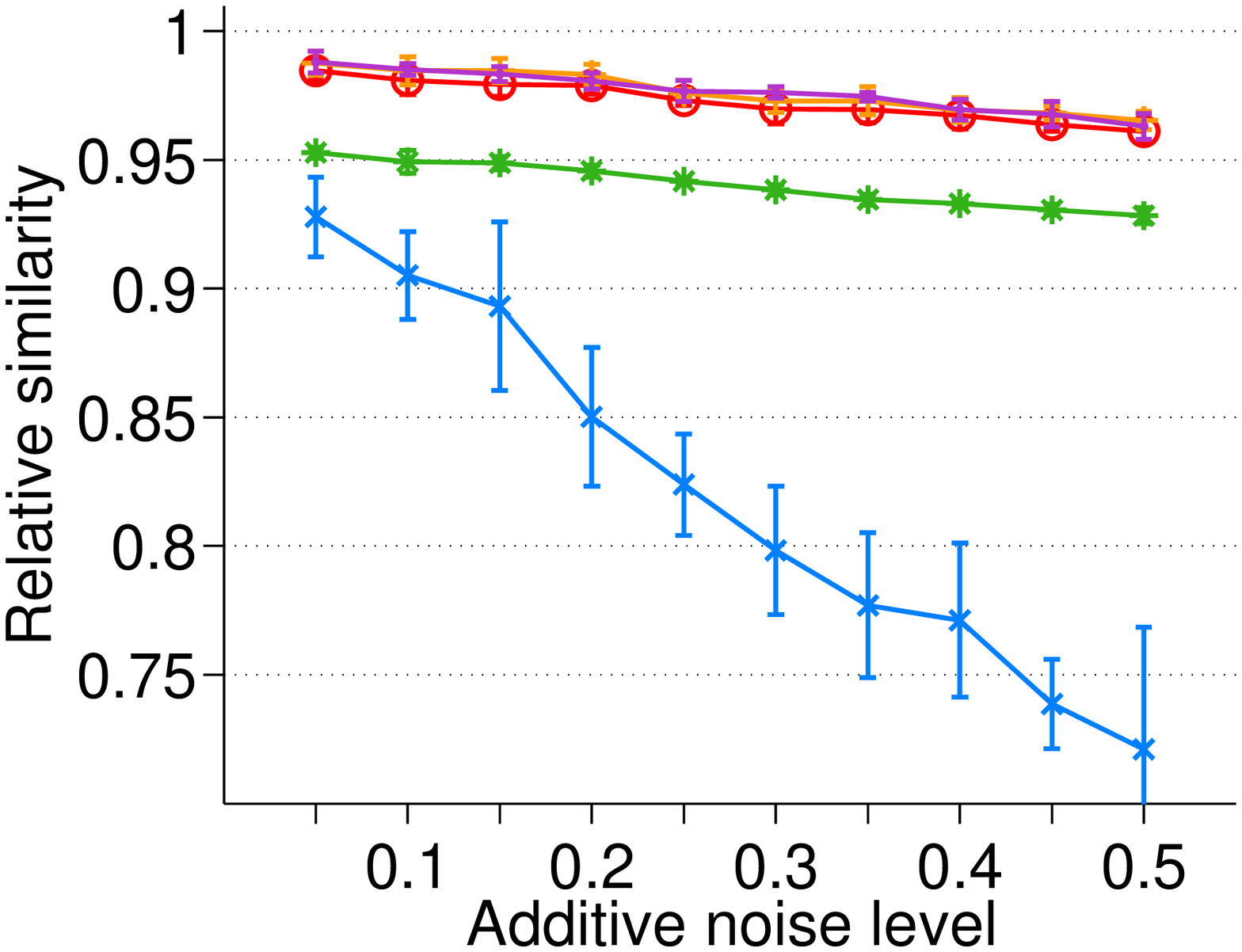}}
  \hspace{\subfigspace}
  \subfigure[Additive noise, compared to the original data\label{fig:synth:additive:orig}]{\includegraphics[width=\subfigwidth]{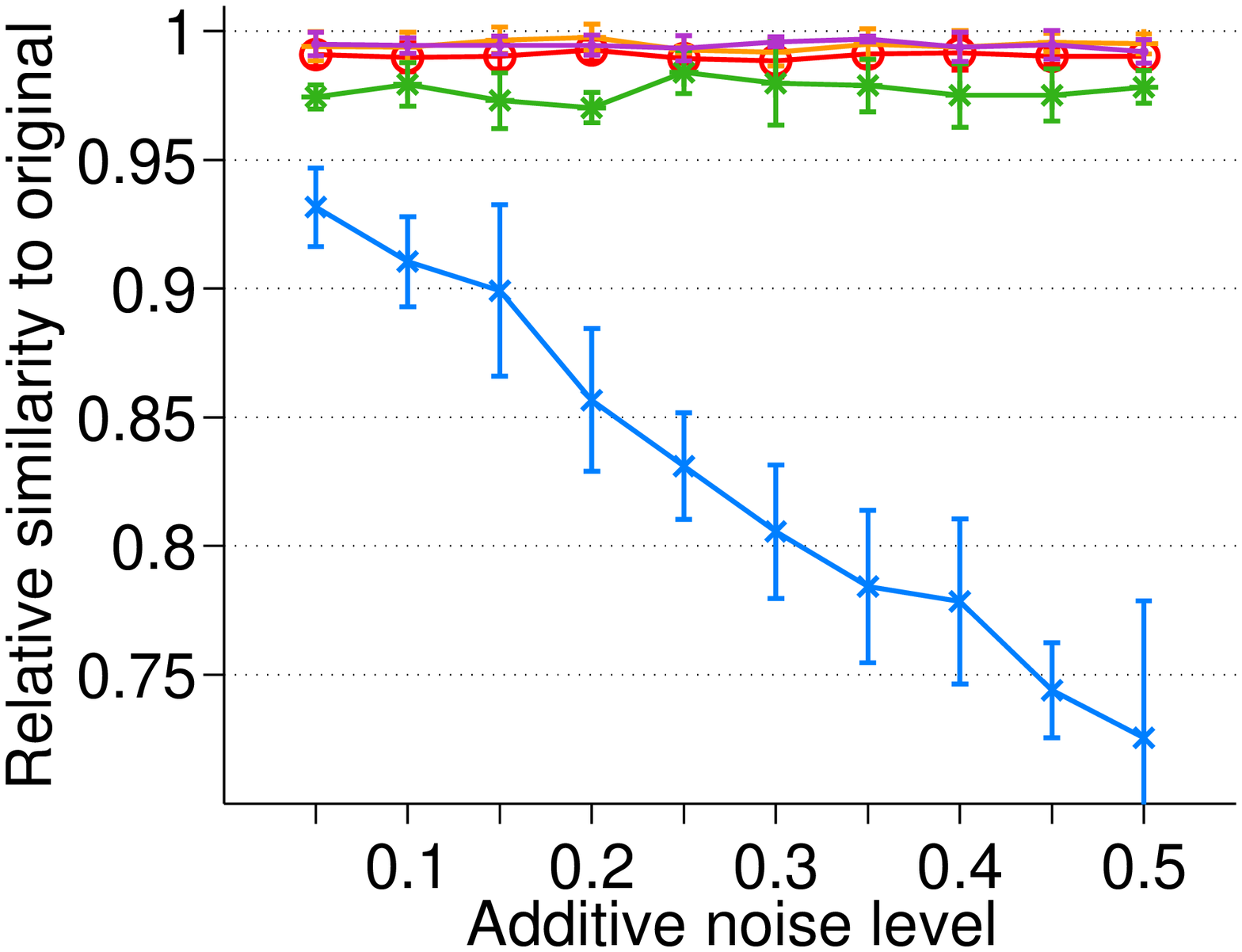}}
  \hspace{\subfigspace}
  \subfigure[Number of clusters\label{fig:synth:k}]{\includegraphics[width=\subfigwidth]{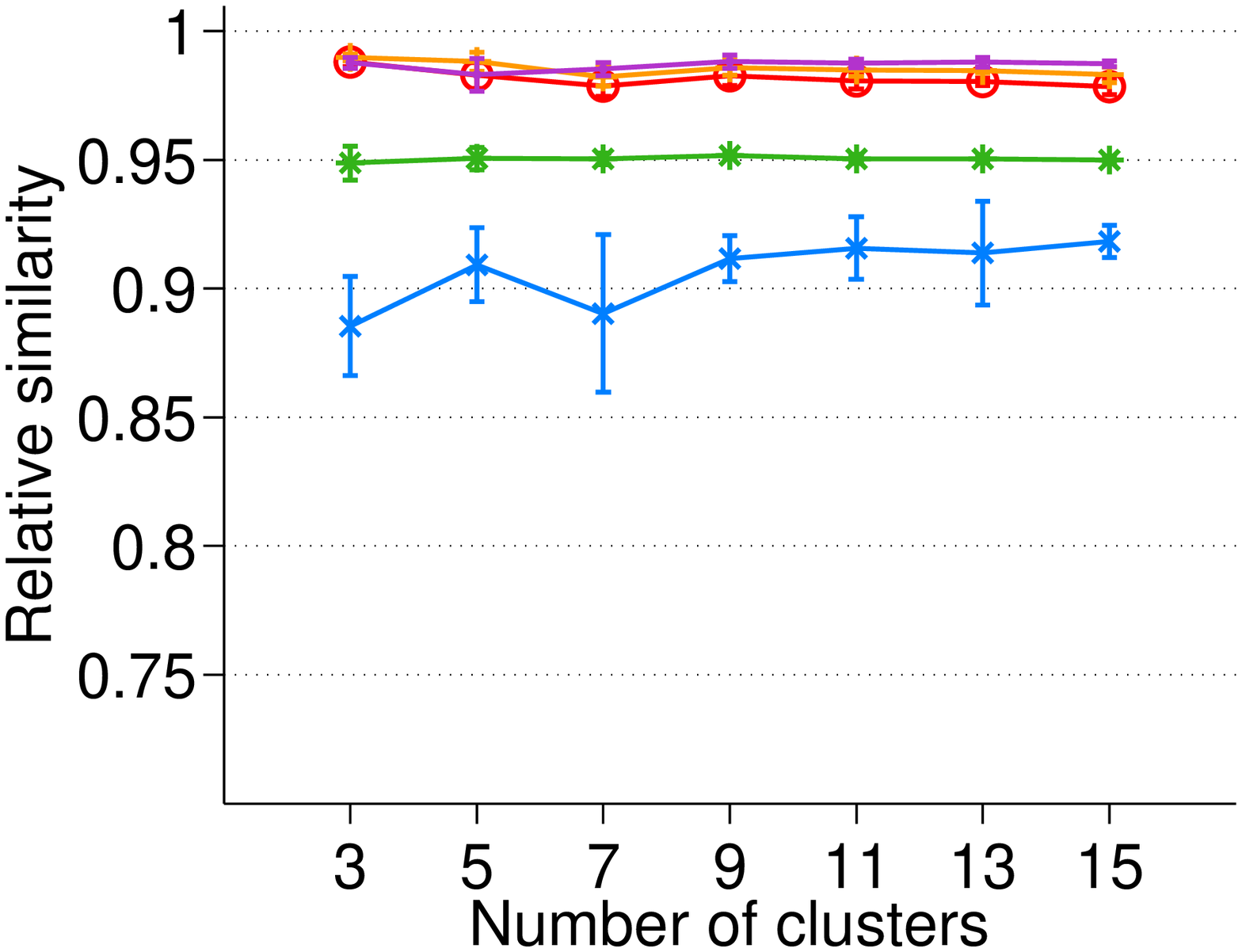}}
  \\[1em]
  \subfigure[Destructive noise\label{fig:synth:destructive}]{\includegraphics[width=\subfigwidth]{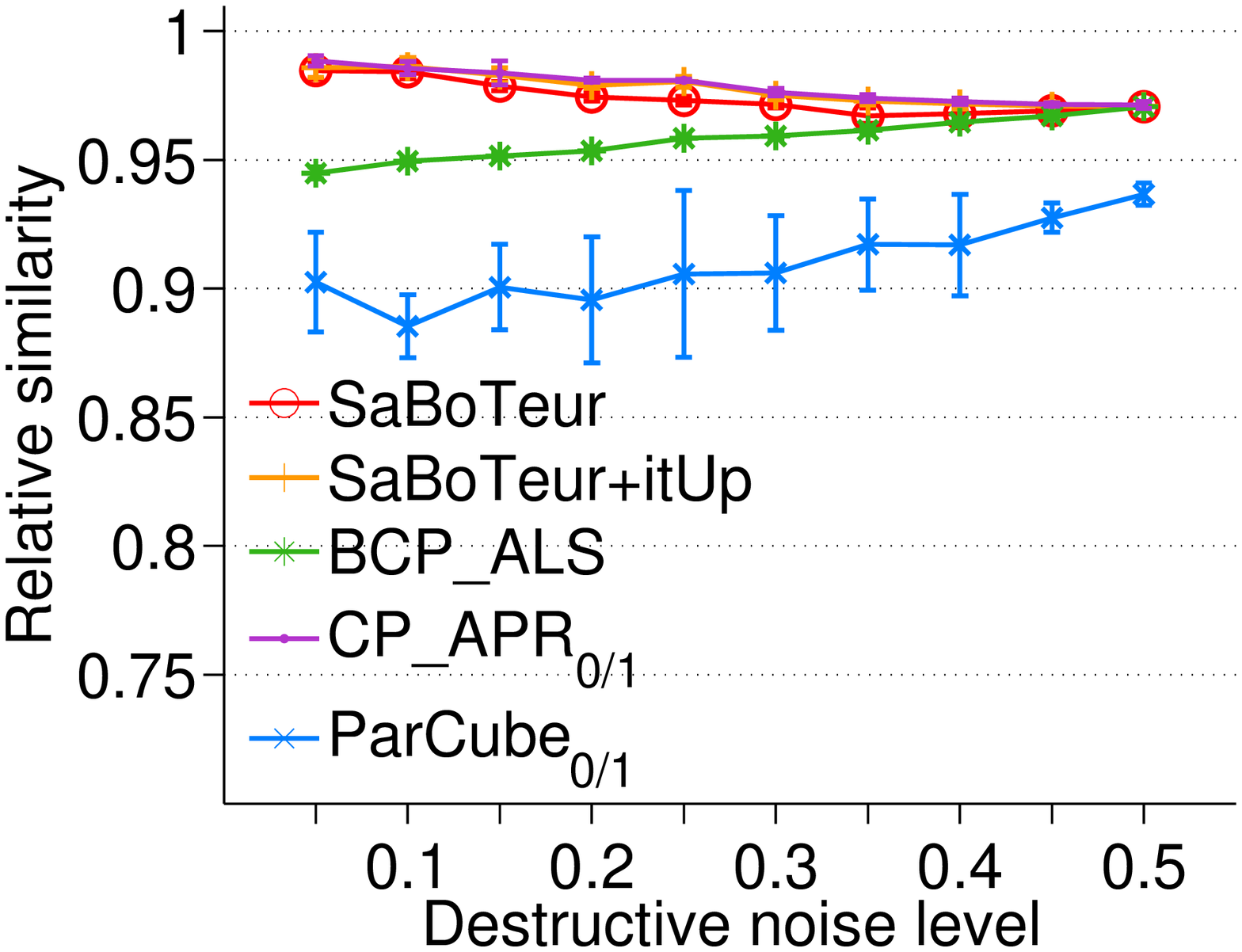}}
  \hspace{\subfigspace}
  \subfigure[Destructive noise, compared to the original data\label{fig:synth:destructive:orig}]{\includegraphics[width=\subfigwidth]{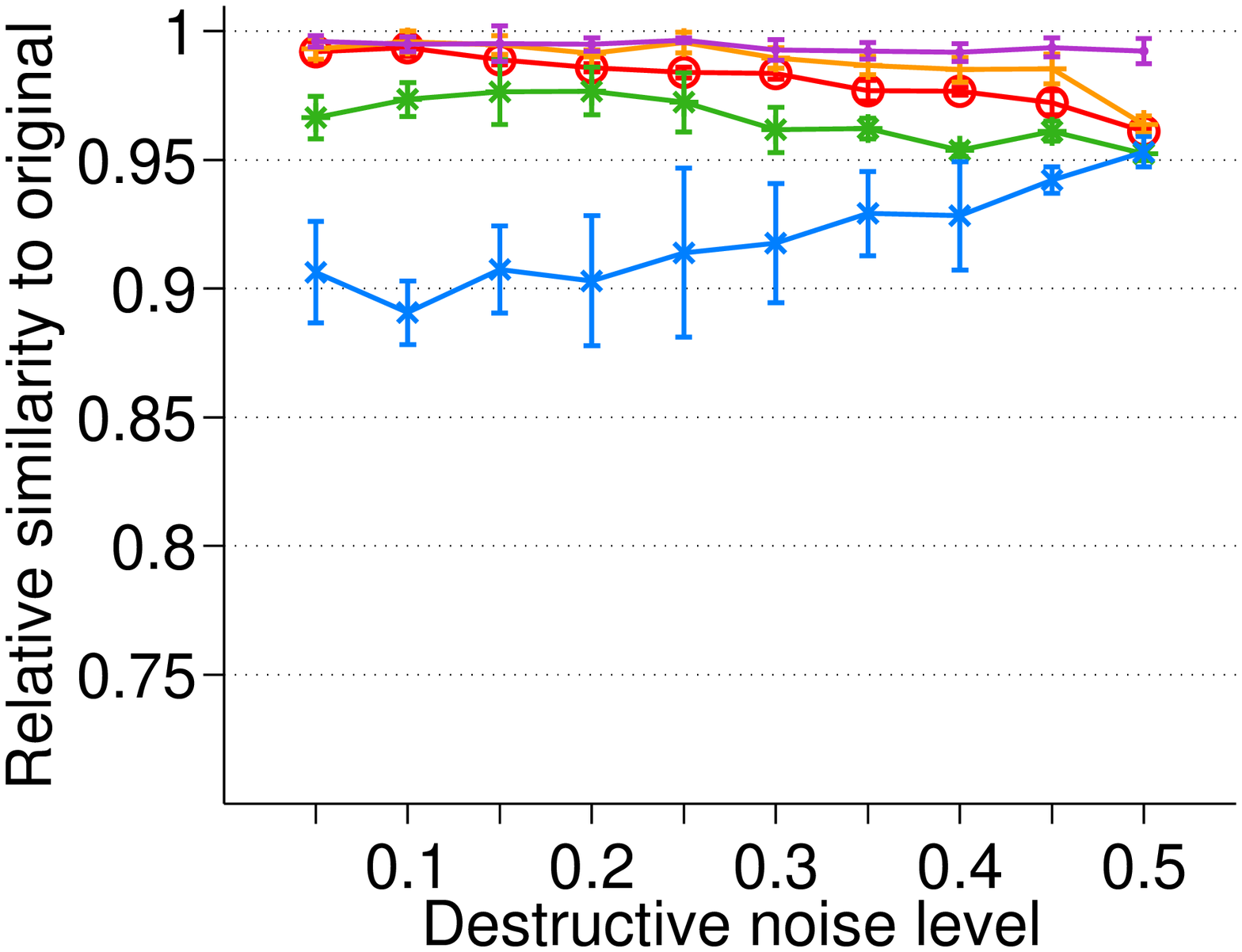}}
  \hspace{\subfigspace}
  \subfigure[Density\label{fig:synth:density}]{\includegraphics[width=\subfigwidth]{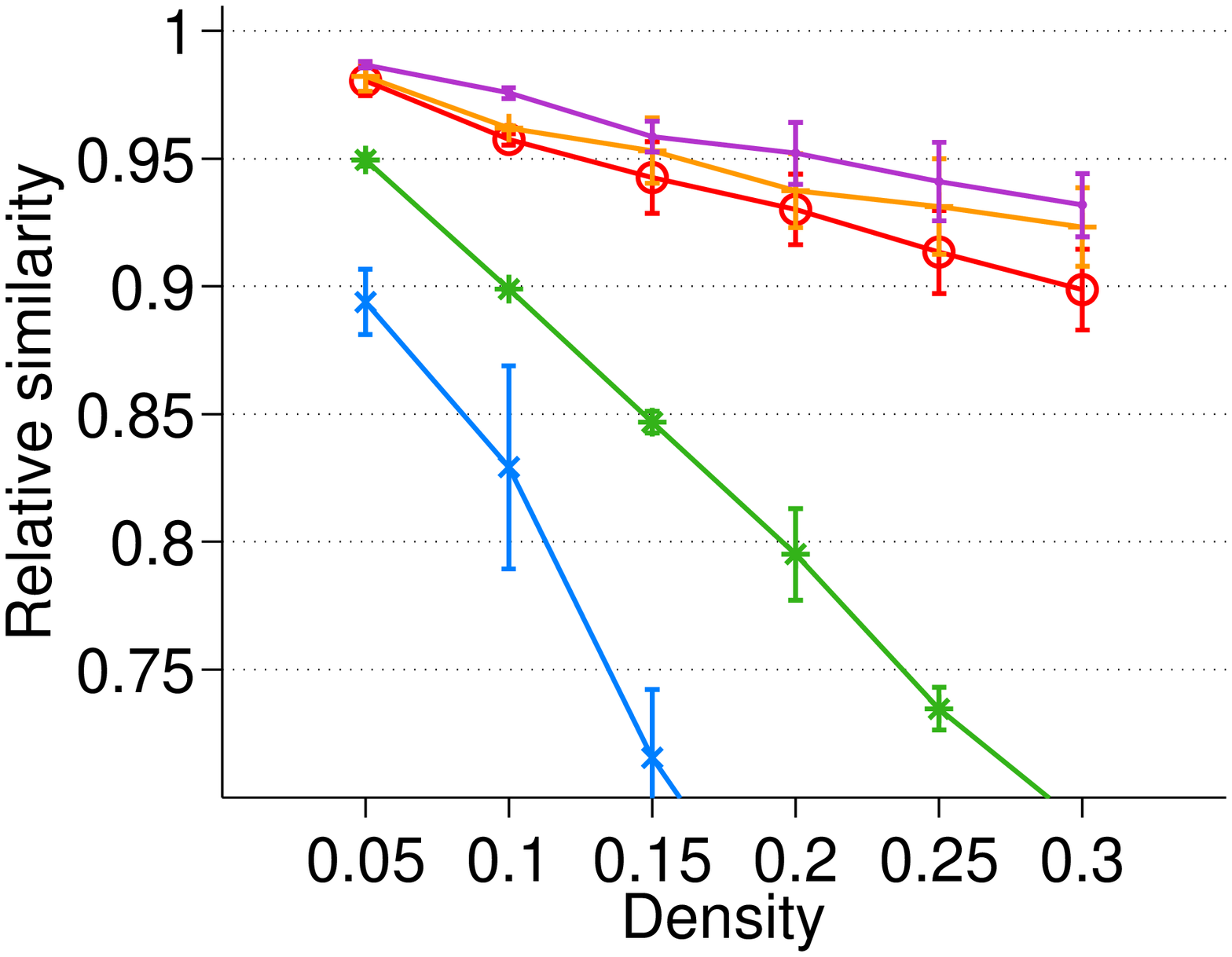}}
  \caption{Results from synthetic experiments when varying different data characteristics. All $y$-axes show the relative similarity.
%; for (b) and (e), the similarity is with respect to the original data; for the others, with respect to the noisy data. 
All markers are mean values over $5$ iterations and the width of the error bars is twice the standard deviation.}
  \label{fig:synth:characteristics}
\end{figure}

\subsubsection{Varying the Number of Clusters}
\label{sec:vary-numb-clust}
The number of clusters varied from $3$ to $15$ with steps of $2$. The results are shown in Figure~\ref{fig:synth:k}.
None of the tested methods show any significant effect to the number of clusters, and the best three methods are still \Saboteur, \SaboteurIU, and \CPAPRr.

\subsubsection{Varying the Density}
\label{sec:varying-density}
The density of the tensor varied from $5\%$ to $30\%$ with steps of $5\%$. The results can be seen in Figure~\ref{fig:synth:density}. All methods perform worse with denser data, with \CPAPRr, \SaboteurIU, and \Saboteur being the best three methods in that order. \BCPALS and \ParCube's results quickly went below $75\%$ similarity, \ParCuber going as low as $55\%$ with $30\%$ density (we omit the worst results from the plots for clarity). 
Note that an increased density implies higher amounts of noise as the level of noise is defined with respect to the number of ones.

\subsubsection{Scalability}
\label{sec:scalability}
We run the scalability tests on a dedicated machine with two Intel Xeon X5650 6-Core CPUs at 2.66GHz and 64GB of main memory. All reported times are wall-clock times. For these experiments, we created a new set of tensors with a default size of \byby{800}{800}{500}, a density of $5\%$, $10\%$ of additive and destructive noise, and $20$ clusters by default. During the experiments, we varied the size, the number of clusters, and the number of threads used for the computation.

We also tried \ParCube, \CPAPR, and \BCPALS with these data sets, but only \ParCube\ was able to finish any of the data sets. However, already with \byby{200}{200}{500} data, the smallest we tried, it took $155.4$ seconds. With $10$ clusters, the least number we used, \ParCube\ finished only after $1319.4$ seconds. Therefore, we omitted these results from the figures.

As can be seen in Figure~\ref{fig:synth:scale:k}, the \Saboteur algorithm scales very well with the number of clusters. This is mostly due to efficient parallelization of the computing of the clusters (c.f.\ below). \SaboteurIU, however, slows down with higher number of clusters. This is to be expected, given that more clusters require more update computations.

\begin{figure}[tb]
  \centering
  \subfigure[Speed vs.\ number of clusters\label{fig:synth:scale:k}]{\includegraphics[width=\subfigwidth]{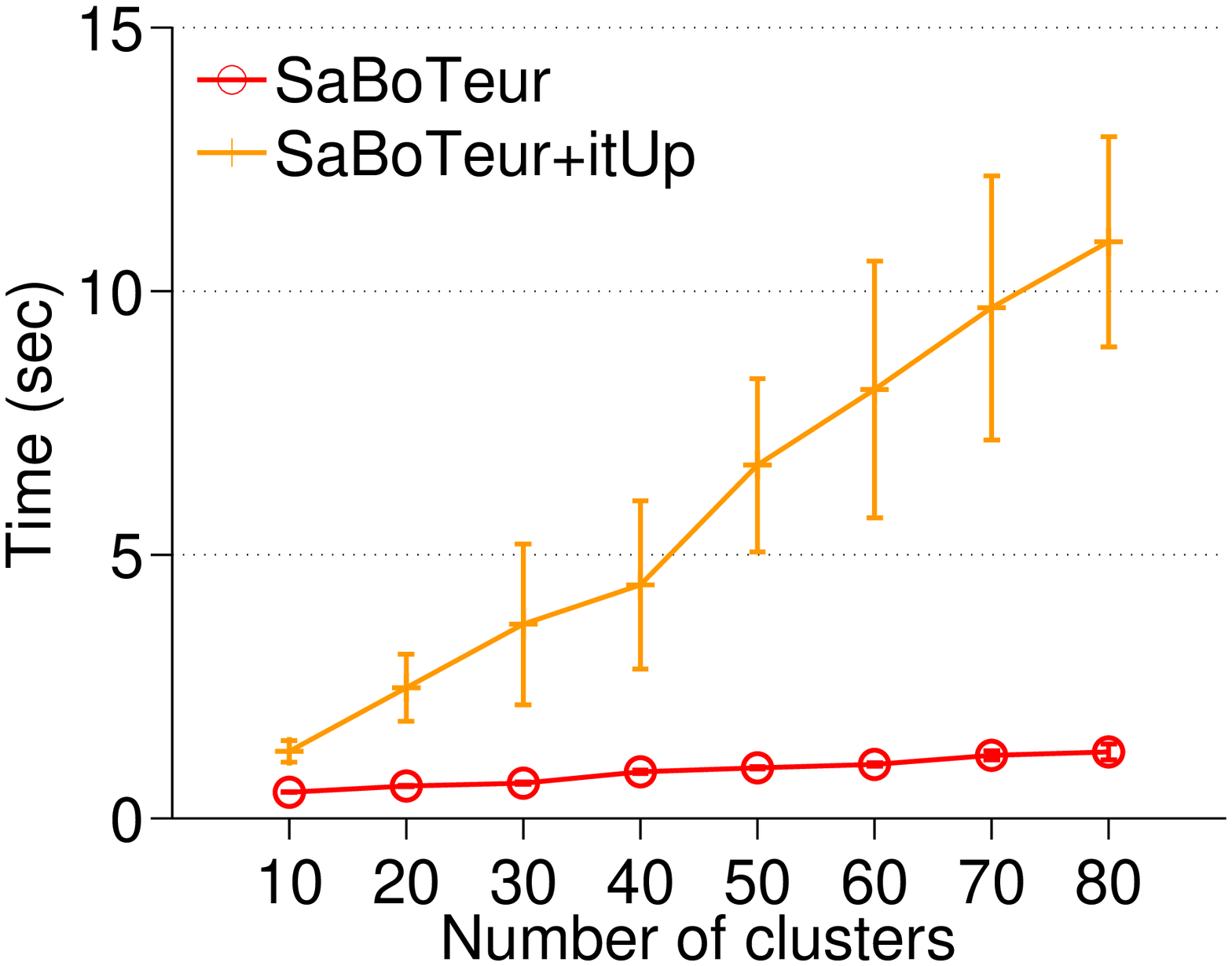}}
  \hspace{\subfigspace}
  \subfigure[Speed vs.\ size of first and second mode\label{fig:synth:scale:n}]{\includegraphics[width=\subfigwidth]{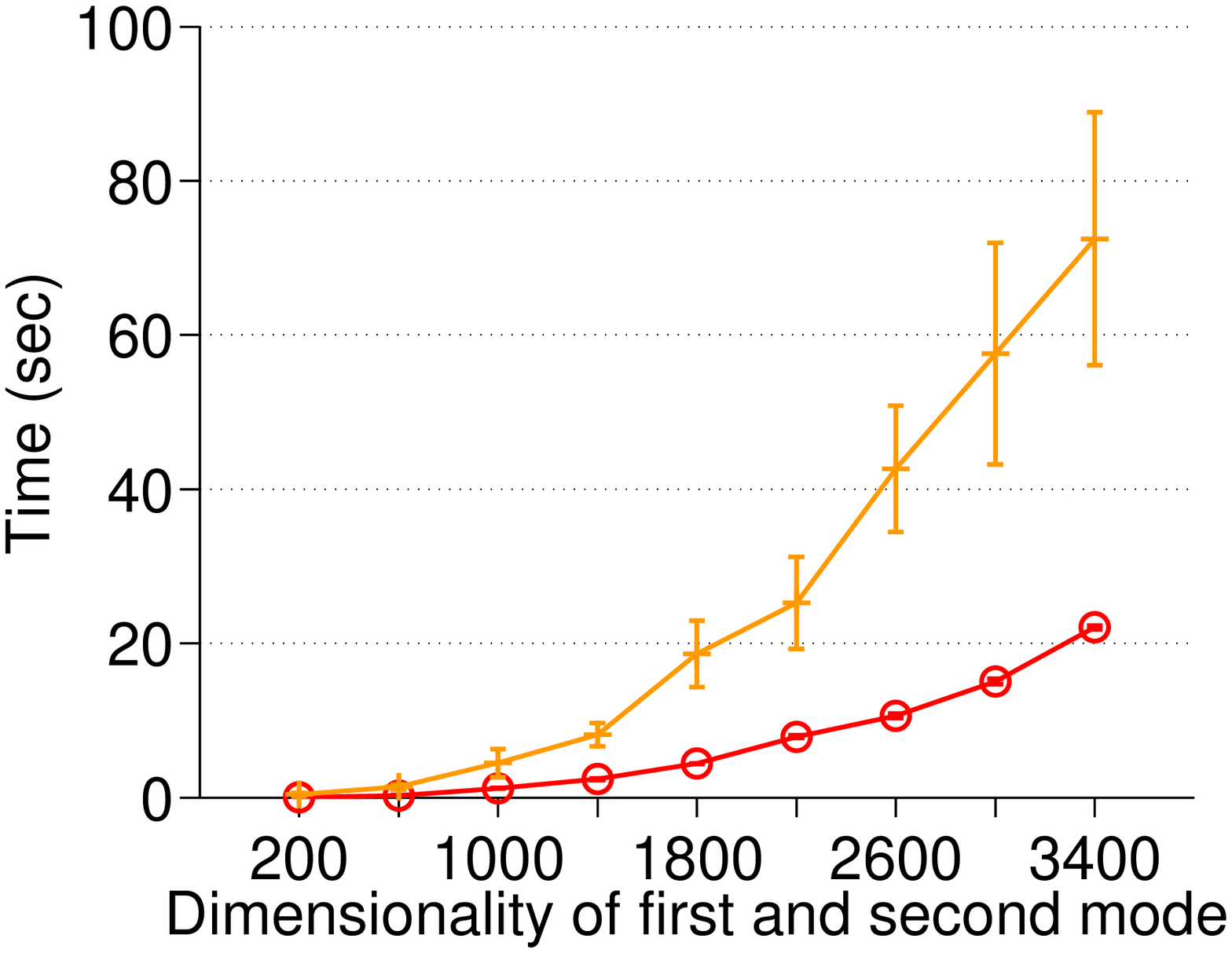}}
  \hspace{\subfigspace}
  \subfigure[Speed vs.\ number of threads\label{fig:synth:scale:threads}]{\includegraphics[width=\subfigwidth]{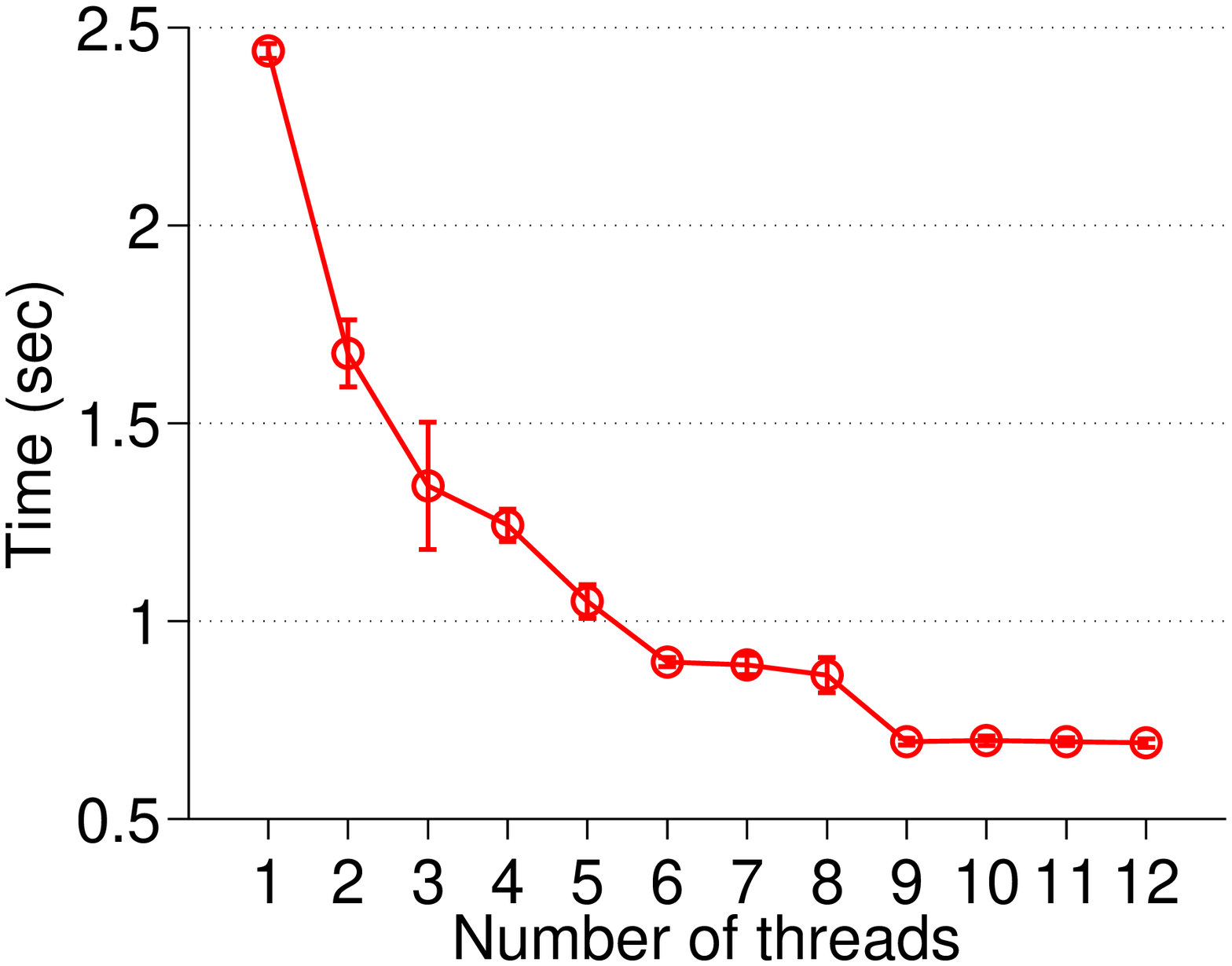}}
  \caption{Results from scalability tests. All times are wall-clock times. Markers are mean values over $5$ iterations and the width of the error bars is twice the standard deviation.}
   \label{fig:synth:scale}
\end{figure}

For the second experiment, we varied the dimensionality of the first and second mode between $200$ and $3400$ with steps of $400$. The results can be seen in Figure~\ref{fig:synth:scale:n}. As we grew both modes simultaneously, the size of the tensor grows as a square (and, as the density was kept constant, so does the number of non-zeros). Given this, the running time of \Saboteur grows as expected, or even slower, while \SaboteurIU is again clearly slower.

In the last scalability experiment (Figure~\ref{fig:synth:scale:threads}), we tested how well \Saboteur parallelizes. As the computer used had $12$ cores, we set that as the maximum number of threads. Yet, after $9$ threads, there were no more obvious benefits. We expect that the memory bus is the limiting factor here.  

\subsubsection{Generalization Tests}
\label{sec:generalization-test}
In our final test with synthetic data, we tested how well \SaboteurIU's clusterings generalize to yet-unseen data. Our hypothesis is that restricting the centroids to rank-1 matrices helps with the overfitting, and hence we compared \SaboteurIU to unrestricted Boolean tensor clustering (BTC), where the centroids are arbitrary binary matrices. Notice that Unrestricted BTC will always obtain at least as good reconstruction error on the training data as \SaboteurIU.

We generated tensors of size \byby{700}{500}{300} with $7$ clusters, $15\%$ of density and different levels of both additive and destructive noise, from $10\%$ till $50\%$. We randomly selected $25\%$ of the frontal slices as the test data and computed the clusterings on the remaining data. We then connected each frontal slice in the test set to its closest centroid. 

\begin{figure}
  \centering
  \subfigure[Test similarity w.r.t. noised (solid lines) and original data (dashed lines)\label{fig:synth:overfitting:error}]{\includegraphics[width=\subfigwidth]{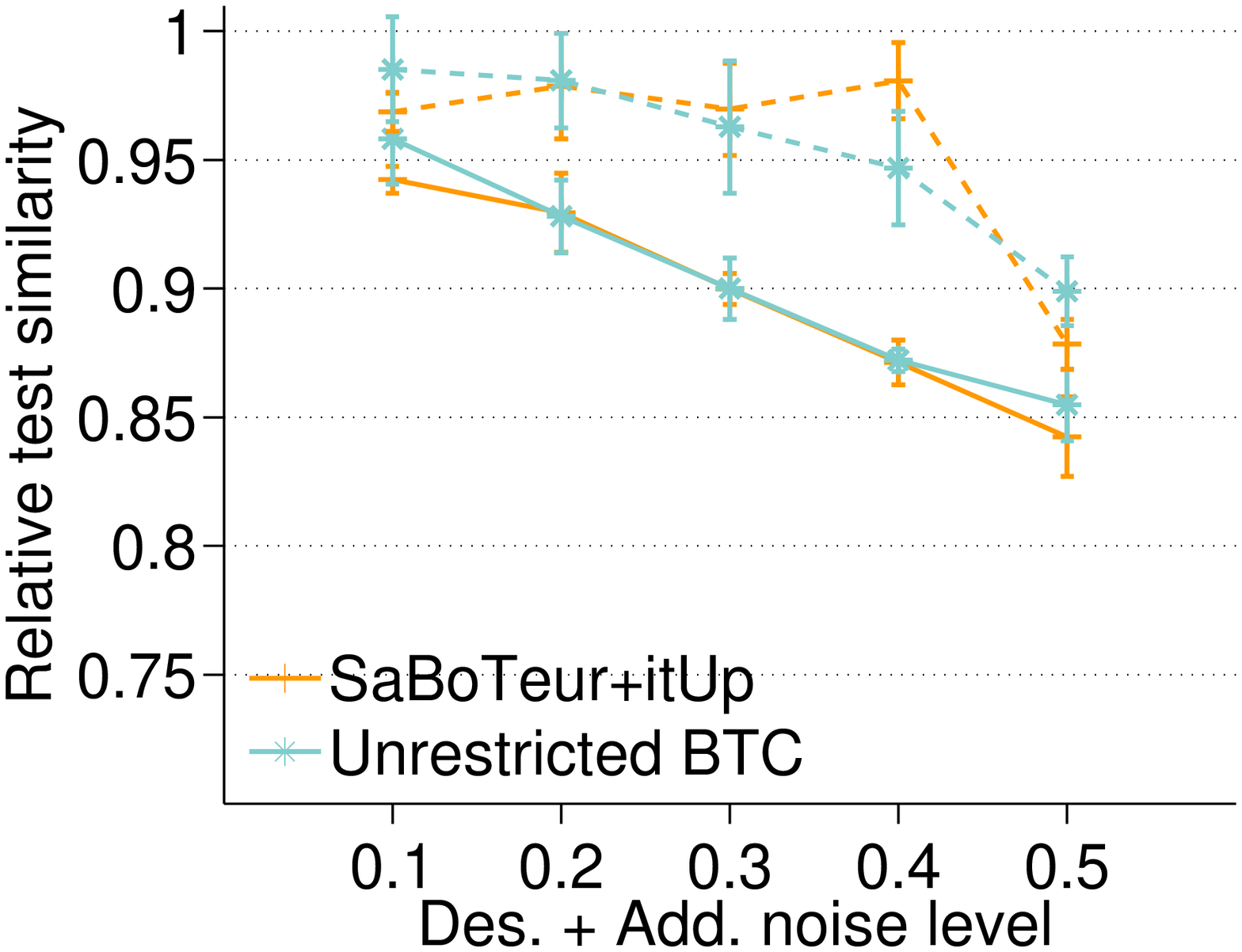}}
  \hspace{\subfigspace}
  \subfigure[Cohen's $\kappa$\label{fig:synth:overfitting:kappa}]{\includegraphics[width=\subfigwidth]{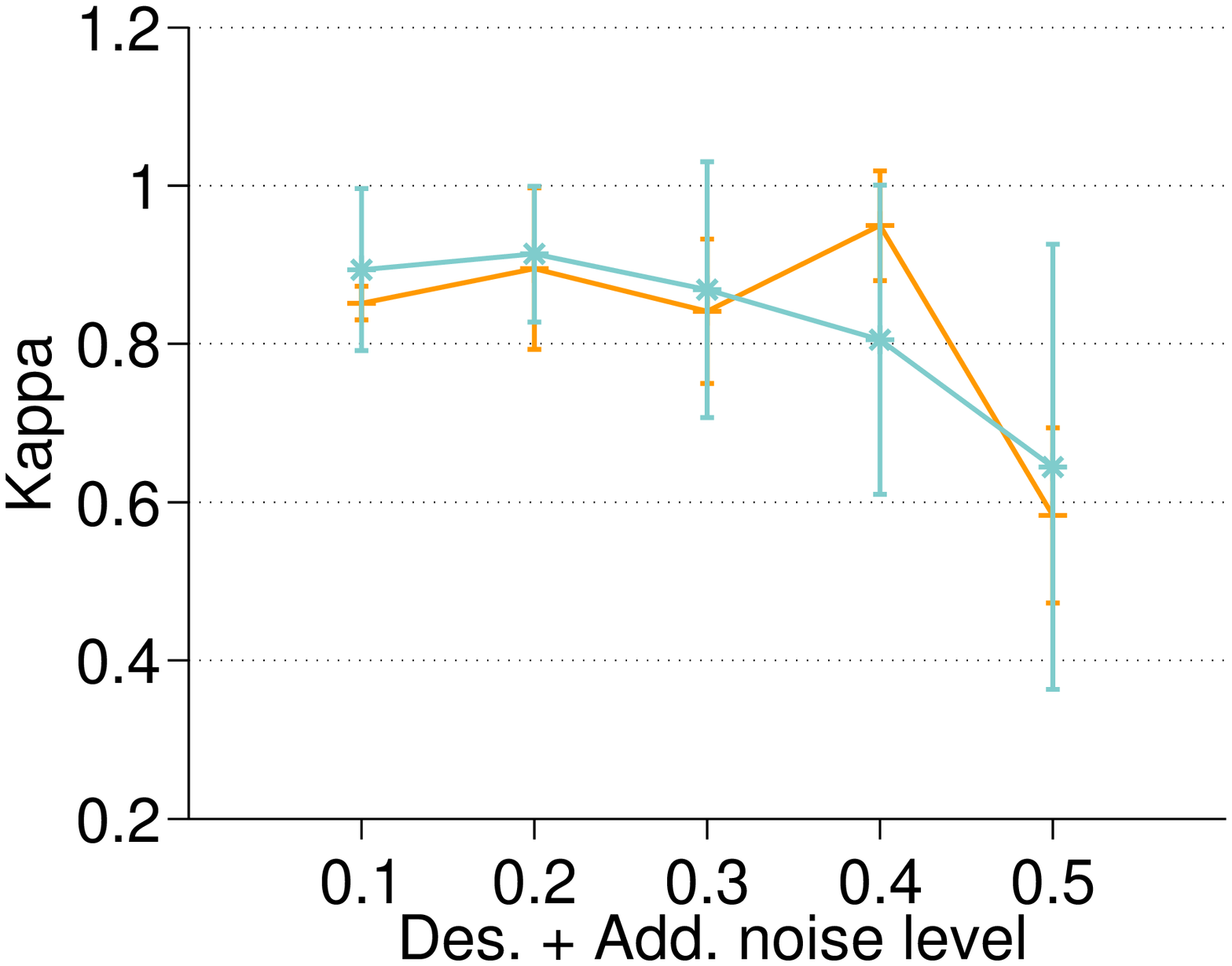}}
  \hspace{\subfigspace}
  \subfigure[Mutual information\label{fig:synth:overfitting:mutual}]{\includegraphics[width=\subfigwidth]{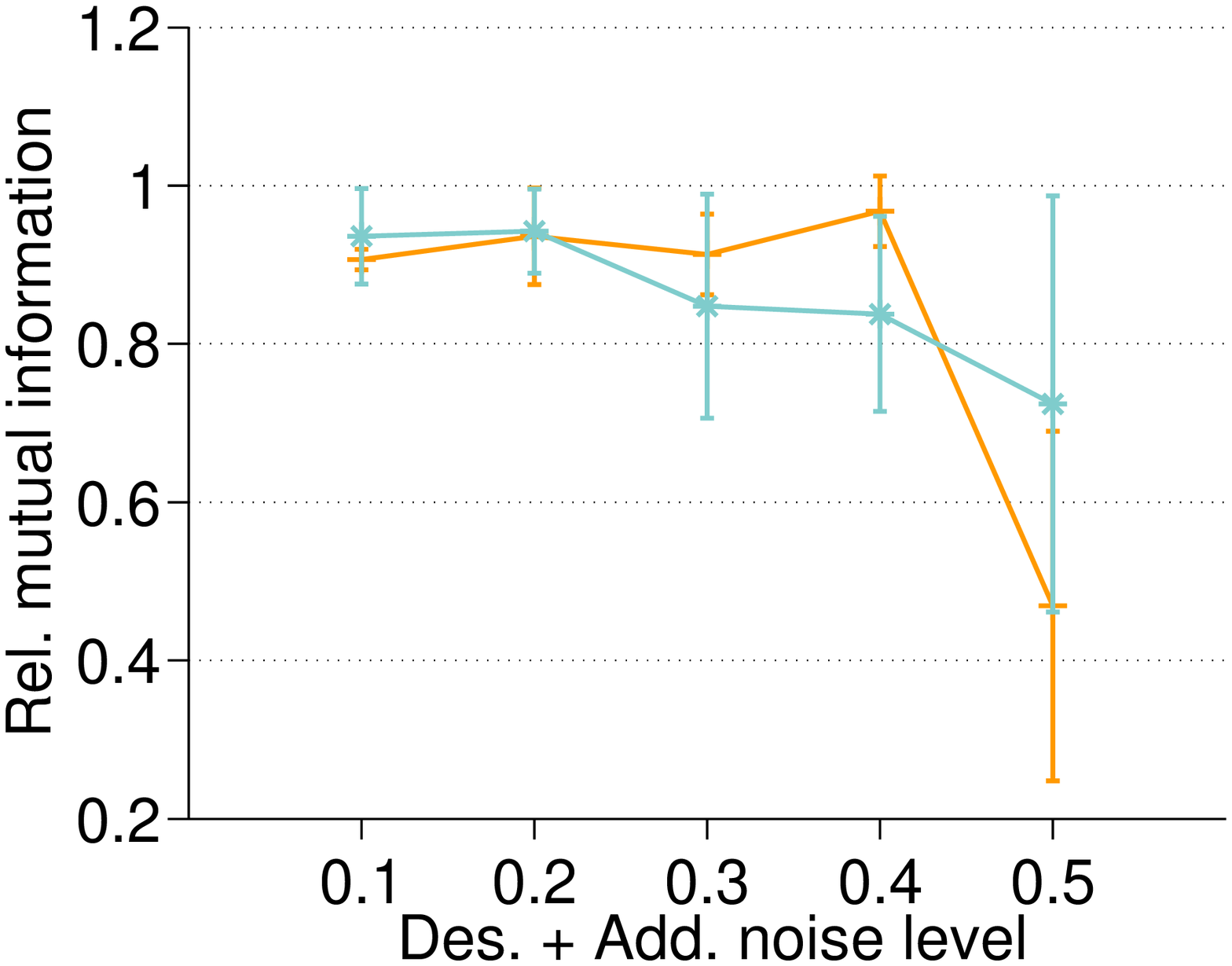}}
  \caption{Evaluation of the generalization quality using (a) reconstruction error on test data, and (b) Kappa statistic and (c) mutual information of cluster labels between the obtained and true labels in the test data. Markers are mean values over 5 different data sets and the width of the error bars is twice the standard deviation.}
  \label{fig:synth:ovefitting}
\end{figure}

The quality of the results was measured in two different ways. First, we used the reconstruction error in the test data (Figure~\ref{fig:synth:overfitting:error}) and computed the overall similarity of expressing the testing data using the centroids. We also used the original (i.e.\ noise-free) frontal slices to assess whether we `model the noise'. The results show that with the noise-free test data (dashed lines), \SaboteurIU is better than the unrestricted BTC with $30$--$40\%$ noise, and the results are reversed with $10\%$ and $50\%$ of noise. With noisy test data, the differences are less.

We also measured the similarity of the cluster labels w.r.t. the ground truth. We used two metrics for that, Cohen's $\kappa$ statistic (Figure~\ref{fig:synth:overfitting:kappa}) and normalized mutual information (Figure~\ref{fig:synth:overfitting:mutual}). Cohen's $\kappa$ takes values from $[-1,1]$, $-1$ meaning complete disagreement and $1$ meaning complete agreement. As Cohen's $\kappa$ requires the clusters' labels to match, we paired the labels using bipartite maximum matching (we used the Hungarian method~\citep{papadimitriou98combinatorial} to compute the matching). The mutual information was normalized by dividing it with the joint entropy; the resulting statistic takes values from $[0,1]$, with $1$ meaning perfect agreement. 

The two measures agree that \SaboteurIU is better at $40\%$ of noise and worse with $10\%$ and $50\%$ of noise. With $10\%$ of noise, the difference is not significant with either metric, and with $50\%$ of noise, the difference is significant only with normalized mutual information. 

Overall, the results show that \SaboteurIU has a small advantage over the unrestricted BTC. Perhaps the clearest sign of the benefits of the rank-1 centroids is the consistently smaller standard deviation \SaboteurIU obtains. This suggests that \SaboteurIU is less susceptible to random variations in the data, which should improve the generalization quality. Indeed, our generalization experiments with the real-world data (Section~\ref{sec:real:overfitting}) support this.

% \subsubsection{Discussion.}
% The synthetic experiments confirm that \Saboteur\ is capable of recovering the latent cluster structure from the synthetic data sets. Arguably the most surprising result of the synthetic experiments was that \Saboteur\ was consistently better than \BCPALS, even though the latter has more freedom to obtain better solutions.

\subsection{Real-World Data}
\label{sec:real-world-data}

We tested \Saboteur also with multiple real-world data sets of varying characteristics. The purpose of these experiments is to verify our findings with synthetic data. To that end, we studied how well \Saboteur (and the other methods) can reconstruct the data, how the methods scale with real-world data, and how \Saboteur generalizes to unknown data. In addition, we also studied the sparsity of the factors, using MDL to select the number of clusters, and how interpretable the results \Saboteur provides are.

\subsubsection{Data Sets}
\label{sec:datasets}
We used nine real-world data sets. The size and density for each data set is listed in Table~\ref{tab:datasets}.
The \Delicious\ data\footnote{\url{http://grouplens.org/datasets/hetrec-2011}\label{fn:hetrec}} \citep{cantador2011workshop} contains user--bookmark--tag tuples from the Delicious social bookmarking system\footnote{\url{http://www.delicious.com}}.
The \Enron\ data\footnote{\url{http://www.cs.cmu.edu/~enron/}} contains information about who sent an e-mail to whom (rows and columns) per months (tubes). 
The \Facebook\ data set\footnote{\url{http://socialnetworks.mpi-sws.org/datasets.html}} \citep{viswanath2009on} contains information about who posted a message on whose wall per week.
The \Lastfm\ data\footref{fn:hetrec} \citep{cantador2011workshop}  consists of artist--user--tag tuples from the \Lastfm online music system\footnote{\url{http://www.last.fm}}.
The \Movielens\ data set and the \Mag\ data set are obtained from the same source\footref{fn:hetrec} \citep{cantador2011workshop}. While \Movielens\ is composed of user--movie--tag tuples, \Mag\ consists of movies--actors--genres tuples.
The \Resolver\ data contains entity--relation--entity tuples from the TextRunner open information extraction algorithm\footnote{\url{http://www.cis.temple.edu/~yates/papers/jair-resolver.html}} \citep{yates09unsupervised}. We used the sample called \ResolverL in~\citet{miettinen11boolean}.
The \TracePort\ data set\footnote{\url{http://www.caida.org/data/passive/passive_2009_dataset.xml}} contains anonymized passive traffic traces (source and destination IP and port numbers) from 2009. 
The \Yago data set consist of  subject--object--relation tuples from the semantic knowledge base \Yago\footnote{\url{http://www.mpi-inf.mpg.de/yago-naga/yago}} \citep{suchanek2007yago}.

As a preprocessing step we removed all the all-zero slices in every mode. In addition, for the \Delicious\ data set, also all slices that had less than 6 entries were purged. For the \Mag\ data, all actors that appeared in less than 5 movies were discarded.

% with content of the datasets  --- too wide unfortunately 
% \begin{table}[t]
%  \centering
%  \small
%\begin{tabular}{@{}l@{\hspace{0.6em}}c@{\hspace{0.6em}}r@{\hspace{0.6em}}r@{\hspace{0.6em}}r@{\hspace{0em}}r@{}}
%    \toprule
%Dataset & Content & \multicolumn{3}{c}{Dimensionality}  & \multicolumn{1}{c}{$d$($*10^{-7}$)} \\ 
%\midrule
%\Delicious & user--bookmark--tag & 1640 & 23584 & 7968 & 8.23 \\ 
%\Enron & sender--recipient--month & 146 & 146 & 38 & 22752.86 \\ 
%\Facebook & sender--recipient--week & 42390 & 39986 & 224 & 16.51 \\ 
%\Lastfm & artist--user--tag & 1892 & 12523 & 9749 & 8.07 \\ 
%\Movielens & user--movie--tag & 2113 & 5908 & 9079 & 4.23 \\ 
%\Mag & movies--actors--genres & 65133 & 10673 & 20 & 162.25 \\ 
%\Resolver & entity--entity--relation & 343 & 360 & 200 & 626.01 \\ 
%\TracePort & source IP--dest IP--port & 10266 & 8622 & 501 & 2.51 \\ 
%\Yago & subject--object--relation & 190962 & 62428 & 25 & 67.35 \\ 
%    \bottomrule 
%\end{tabular}
% \caption{Summary of datasets}
%   \label{tab:datasets}
% \end{table}

 \begin{table}[tb]
\topcaption{Size and density of the real-world datasets.}
   \label{tab:datasets}
  \centering
  \small
\begin{tabular}{@{}l@{\hspace{3em}}rrrr@{}}
    \toprule
Dataset & %\multicolumn{3}{c}{Dimensionality}  
Rows & Columns & Tubes & Density ($10^{-7}$) \\ 
\midrule
\Delicious & 1640 & 23584 & 7968 & 8.23 \\ 
\Enron & 146 & 146 & 38 & 22752.86 \\ 
\Facebook  & 42390 & 39986 & 224 & 16.51 \\ 
\Lastfm  & 1892 & 12523 & 9749 & 8.07 \\ 
\Movielens & 2113 & 5908 & 9079 & 4.23 \\ 
\Mag  & 10197 & 10673 & 20 & 1036.37 \\ 
\Resolver  & 343 & 360 & 200 & 626.01 \\ 
\TracePort  & 10266 & 8622 & 501 & 2.51 \\ 
\Yago  & 190962 & 62428 & 25 & 67.35 \\ 
    \bottomrule 
\end{tabular}
 \end{table}

%%% Local Variables: 
%%% mode: latex
%%% TeX-master: "main"
%%% End: 

\subsubsection{Reconstruction Accuracy}
\label{sec:real:accuracy}
As all real-world data sets are very sparse, we report the errors, not the similarity, for these experiments. The error is measured as the number of disagreements, when the reconstructed tensor is also binary, or as squared tensor Frobenius distance, when the reconstructed tensor is real-valued (not-rounded versions of \CPAPR and \ParCube). The results can be seen in Table~\ref{tab:real:results}. The results for \CPAPRr and \ParCuber for all data sets except \Enron and \Resolver are obtained by sampling: If $\abs{\tens{X}}$ is the number of non-zeros in data $\tens{X}$, we sampled $200\abs{\tens{X}}$ locations of $0$s, and computed the error the methods make for every non-zero and for the sampled zero elements. The sampled results were then extrapolated to the size of the full tensor. We had to use sampling, as the factor matrices were too dense to allow reconstructing the full tensor. We will discuss more on density below.

\begin{sidewaystable}
%% Sidewaystable doesn't calculate the size of the box correctly and/or rotates from the bottom of the table. This vspace is needed to push the table back in the page.
%\vspace*{10cm}
\topcaption{Reconstruction errors on real data sets, measured as squared tensor Frobenius distance. `---' means that the algorithm was not able to finish.}
\label{tab:real:results} 

\small

\centering

\begin{tabular}{@{}l
	r@{\hspace{0.6em}}r@{\hspace{0.6em}}r@{\hspace{0.6em}}r
	r@{\hspace{0.6em}}r@{\hspace{0.6em}}r@{\hspace{0.6em}}r
	r@{\hspace{0.6em}}r@{\hspace{0.6em}}r@{\hspace{0.6em}}r
	r@{\hspace{0.6em}}r@{}}
    \toprule
\multicolumn{1}{r}{}	&	\multicolumn{ 4}{c}{\Delicious}	&							\multicolumn{ 4}{c}{\Enron}	&							\multicolumn{ 4}{c}{\Lastfm}	&					\multicolumn{ 2}{c}{\Mag}					\\
\cmidrule(r){2-5}																												
\cmidrule(rl){6-9}																												
\cmidrule(rl){10-13}																												
\cmidrule(l){14-15}																												
						
\multicolumn{1}{r}{$k=$}	&	\multicolumn{1}{c}{$7$}	&	\multicolumn{1}{c}{\hspace{-0.6em}$10$}	&	\multicolumn{1}{c}{\hspace{-0.6em}$15$}	&	\multicolumn{1}{c}{\hspace{-0.6em}$30$}	&	\multicolumn{1}{c}{$5$}	&	\multicolumn{1}{c}{\hspace{-0.6em}$10$}	&	\multicolumn{1}{c}{\hspace{-0.6em}$12$}	&	\multicolumn{1}{c}{\hspace{-0.6em}$15$}	&	\multicolumn{1}{c}{$8$}	&	\multicolumn{1}{c}{$15$}	&	\multicolumn{1}{c}{$20$}	&	\multicolumn{1}{c}{$30$}	&	\multicolumn{1}{c}{$5$}	&	\multicolumn{1}{c}{$7$}			\\
    \midrule																															
\Saboteur	&	$253608$	&	$253546$	&	$253447$	&	$253004$	&	$1811$	&	$1779$	&	$1769$	&	$1753$	&	$195905$	&	$195735$	&	195570	&	$195278$	&	$224322$	&	$223580$			\\
\SaboteurIU	&	$253153$	&	$253293$	&	$252738$	&	$252859$	&	$1793$	&	$1756$	&	$1750$	&	$1735$	&	$186072$	&	$185713$	&	185651	&	$185399$	&	$223869$	&	$223489$			\\
\BCPALS	&	---	&	---	&	---	&	---	&	$1850$	&	$1850$	&	$1850$	&	$1850$	&	---	&	---	&	---	&	---	&	---	&	---			\\
\CPAPR	&	$253558$	&	$253516$	&	$253433$	&	$253196$	&	$1718$	&	$1631$	&	$1598$	&	$1560$	&	$184513$	&	$184038$	&	$183489$	&	$182659$	&	$224996$	&	$224696$			\\
\CPAPRr	&	$253653$	&	$253653$	&	$253652$	&	$253639$	&	$1838$	&	$1817$	&	$1811$	&	$1781$	&	$180766$	&	$177527$	&	$174491$	&	$170939$	&	$225580$	&	$225155$			\\
\ParCube	&	$367521$	&	$373710$	&	$363581$	&	$418252$	&	$2137$	&	$2273$	&	$2352$	&	$2270$	&	$2745259$	&	$2879381$	&	$2856941$	&	$2770361$	&	$375290$	&	$322121$			\\
\ParCuber	&	$247129$	&	$246566$	&	$246016$	&	$241850$	&	$1802$	&	$1744$	&	$1751$	&	$1690$	&	$133057$	&	$139207$	&	$119461$	&	$131322$	&	$222133$	&	$221933$			\\
\walknmerge	&	---	&	$251034$	&	---	&	---	&	---	&	---	&	$1753$	&	---	&	---	&	---	&	---	&	---	&	---	&	---			\\
    \bottomrule 
\end{tabular}

\vspace{0.5em}
\begin{tabular}{@{}l@{\hspace{0.6em}}
	r@{\hspace{0.6em}}r@{\hspace{0.6em}}r@{\hspace{0.6em}}r
	r@{\hspace{0.6em}}r@{\hspace{0.6em}}r@{\hspace{0.6em}}r
	r@{\hspace{0.6em}}r@{\hspace{0.6em}}r@{\hspace{0.6em}}r
	r@{\hspace{0.6em}}r@{\hspace{0.6em}}
	r@{}}
    \toprule
\multicolumn{1}{r}{}	&	\multicolumn{ 4}{c}{\Movielens}							&	\multicolumn{ 4}{c}{\Resolver}							&	\multicolumn{ 4}{c}{\TracePort}									&	\multicolumn{ 2}{c}{\Facebook}	&	\multicolumn{ 1}{c}{\Yago} \\ 
\cmidrule(r){2-5}
\cmidrule(rl){6-9}			
\cmidrule(rl){10-13}			
\cmidrule(rl){14-15}																												
\cmidrule(rl){16-16}																												
	
\multicolumn{1}{r}{$k=$}	&	\multicolumn{1}{c}{$5$}	&	\multicolumn{1}{c}{$15$}	&	\multicolumn{1}{c}{$20$}	&	\multicolumn{1}{c}{$30$}	&	\multicolumn{1}{c}{$5$}	&	\multicolumn{1}{c}{$10$}	&	\multicolumn{1}{c}{$15$}	&	\multicolumn{1}{c}{$30$}	&	\multicolumn{1}{c}{$5$}	&	\multicolumn{1}{c}{$10$}	&	\multicolumn{1}{c}{$15$}	&	\multicolumn{1}{c}{$30$}	&	\multicolumn{1}{c}{$15$}	&	\multicolumn{1}{c}{$20$}	&	\multicolumn{1}{c}{$7$} \\ 
    \midrule																														
\Saboteur	&	$56761$	&	$56712$	&	$56678$	&	$56571$	&	$1530$	&	$1509$	&	$1485$	&	$1455$	&	$11085$	&	$11004$	&	$10923$	&	$10776$	&	$626999$	&	$626776$	&	$1957737$ \\ 
\SaboteurIU	&	$47933$	&	$47479$	&	$47280$	&	$47316$	&	$1522$	&	$1503$	&	$1457$	&	$1443$	&	$10990$	&	$10961$	&	$10913$	&	$10673$	&	---	&	---	&	--- \\ 
\BCPALS	&	---	&	---	&	---	&	---	&	$1624$	&	$1624$	&	$1626$	&	$1632$	&	---	&	---	&	---	&	---	&	---	&	---	&	--- \\ 
\CPAPR	&	$47584$	&	$46927$	&	$46555$	&	$46287$	&	$1519$	&	$1509$	&	$1489$	&	$1457$	&	$11140$	&	$11114$	&	$11082$	&	$11027$	&	$626343$	&	$626202$	&	--- \\ 
\CPAPRr	&	$47178$	&	$44655$	&	$43614$	&	$42609$	&	$1545$	&	$1545$	&	$1538$	&	$1539$	&	$11110$	&	$11059$	&	$10893$	&	$10617$	&	$626945$	&	$626945$	&	--- \\ 
\ParCube	&	$252719$	&	$463025$	&	$454826$	&	$449534$	&	$1784$	&	$1762$	&	$1798$	&	$1939$	&	$19803$	&	$30793$	&	$31107$	&	$29796$	&	$750891$	&	$797534$	&	$619476763$ \\ 
\ParCuber	&	$41223$	&	$38098$	&	$37915$	&	$37393$	&	$1507$	&	$1471$	&	$1450$	&	$1370$	&	$10708$	&	$9913$	&	$10097$	&	$9647$	&	$619939$	&	$620894$	&	$4793390$ \\ 
\walknmerge	&	---	&	$46807$	&	---	&	---	&	---	&	---	&	$1534$	&	---	&	---	&	---	&	$10679$	&	---	&	---	&	---	&	--- \\ 
    \bottomrule 
\end{tabular}

\end{sidewaystable}

%%% Local Variables: 
%%% mode: latex
%%% TeX-master: "main"
%%% End: 

The smallest reconstruction error is obtained by \ParCuber in almost all experiments, \Yago being a notable exception. Remember, however, that \ParCuber returns a non-negative tensor CP decomposition that is afterwards rounded to binary tensor, that is, it is neither clustering nor Boolean, and hence together with \CPAPRr is expected to be better than \Saboteur. Also, without the rounding, \ParCube is often the worst method by a large margin. \CPAPR benefits much less from the rounding, \CPAPRr being comparable to \Saboteur. 

In those data sets where we were able to run \BCPALS, its results were consistently worse than \Saboteur's results, despite it solving more relaxed problem. We were unable to get \walknmerge to finish within a reasonable time with most data sets: it found rank-1 tensors sufficiently quickly, but took too much time to select the top ones from there (see~\citet{erdos13walknmerge} for explanation on how \walknmerge finds a Boolean CP decomposition). When it did finish, however, it was slightly better than \SaboteurIU or \Saboteur. 

Finally, the iterative updates of \SaboteurIU consistently improved the results compared to \Saboteur, but did so with significant increase to the running time. It seems that with most data sets, the iterative updates are not necessarily worth the extra wait.

\subsubsection{Sparsity of the Factors}
\label{sec:real:sparsity}
An important question on the practical feasibility of the tensor decomposition and clustering algorithms is the density of the factor matrices. Too dense factor matrices increase the computational complexity and also the storage requirements. Furthermore, multiplying the dense factors together becomes prohibitively expensive, making it impossible to fully re-construct the tensor. This is the reason why, for example, we had to use sampling to compute the rounded representations of \CPAPR and \ParCube.

We studied the sparsity of the factors using the \Delicious data. In Table~\ref{tab:density} we report the total densities of the matrices $\matr{A}$ and $\matr{B}$, that is, $(\abs{\matr{A}} + \abs{\matr{B}})/k(n+m)$ for \by{n}{k} and \by{m}{k} factors. 

 \begin{table}[tb]
   \topcaption{Total density of the factor matrices $\matr{A}$ and $\matr{B}$ for different $k$ on the \Delicious data.}
   \label{tab:density}
   \centering
   \small
   \begin{tabular}{@{}lrrrr@{}}
     \toprule
\multicolumn{1}{r}{$k=$} &\multicolumn{1}{c}{7} & \multicolumn{1}{c}{10} & \multicolumn{1}{c}{15} & \multicolumn{1}{c}{30} \\ 
     \midrule
 \Saboteur & 0.00016 & 0.00031 & 0.00029 & 0.00029 \\
% \SaboteurIU & 0.0000004 & 0.0000010 & 0.0000008 & 0.0000016 \\ 
 \CPAPR & 0.20219 & 0.14786 & 0.10312 & 0.05396 \\ 
 \ParCube & 0.28186 & 0.21069 & 0.16658 & 0.09900 \\ 
 \walknmerge & ---& 0.00022 &--- & --- \\ 
% \Saboteur & 0.0001574 & 0.0003114 & 0.0002888 & 0.0002874 \\
% \SaboteurIU & 0.0000004 & 0.0000010 & 0.0000008 & 0.0000016 \\ 
% \CPAPR & 0.2021884 & 0.1478631 & 0.1031161 & 0.0539632 \\ 
% \ParCube & 0.2818631 & 0.2106882 & 0.1665768 & 0.0990023 \\ 
% \walknmerge & ---& 0.0002225 &--- & --- \\ 
     \bottomrule 
\end{tabular}
 
 \end{table}

It is obvious that the Boolean methods, \Saboteur
%, \SaboteurIU, 
and \walknmerge, are orders of magnitude sparser than the continuous methods (\CPAPR and \ParCube). This was to be expected as similar behaviour has already been observed with Boolean matrix factorizations \citep{miettinen10sparse}. 

We did not compare the third-mode factor matrices for densities. For \Saboteur, the clustering assignment matrix $\matr{C}$ has density $1/l$ that depends only on the number of frontal slices $l$; obtaining density less than that requires having rows of $\matr{C}$ that have no non-zeros.

\subsubsection{Scalability with Real-World Data}
\label{sec:real:scale}
Table~\ref{tab:real:timings} shows the wall-clock times for \Movielens and \Resolver with $k=15$. \Movielens is one of the sparsest 
data sets, while \Resolver is one of the densest ones, and hence these two show us how the relative speed of the methods changes as the density changes. The other data sets generally followed the behaviour we report here, adjusting to their density. 

  \begin{table}[tb]
  \topcaption{Average wall-clock running times in seconds for real-world data sets. Averages are over $5$ re-starts for \Movielens and $10$ for \Resolver. Both data sets used $k=15$ clusters.}
    \label{tab:real:timings}
   \centering
   \small
 \begin{tabular}{@{}lrr@{}}
     \toprule
  & \Movielens & \Resolver  \\ 
 \midrule
 \Saboteur & 45.6	&	0.02 \\ 
 \SaboteurIU &	303.6	&	0.11 \\ 
 \CPAPR &	196.6	&	22.00 \\ 
 \ParCube &	24.8	&	 7.68 \\ 
 \walknmerge &	207.5	&	2.15 \\ 
     \bottomrule 
 \end{tabular}
  \end{table}

%%% Local Variables: 
%%% mode: latex
%%% TeX-master: "main"
%%% End: 

With the\Movielens data, \ParCube is the fastest of the methods, followed by \Saboteur. \CPAPR, \walknmerge, and \SaboteurIU are significantly slower than the two. With the denser \Resolver data, however, the order is changed, with \Saboteur and \SaboteurIU being an order of magnitude faster than \walknmerge, which still is faster than \ParCube or \CPAPR. The density of the data does not affect the speed of \Saboteur, while it does have a significant effect to \CPAPR and \ParCube. \SaboteurIU is less affected by density, and more affected by the size of the data, than the other methods.

\subsubsection{Generalization with Real-World Data}
\label{sec:real:overfitting}
We repeated the generalization test we did with synthetic data with three real-world data sets -- \Movielens, \Resolver, and \TracePort\ -- keeping $85\%$ of the frontal slices as training data while the remaining $15\%$ we used for testing. The results can be seen in Table~\ref{tab:real:overfitting}. We used $k=15$ clusters for \Movielens\ and \TracePort, and $10$ clusters for the \Resolver\ data. On each data set we conducted $10$ repetitions of the training phase, and selected the one with the smallest training error.

%  \begin{table}[tb]
%   \centering
%   \small
% \begin{tabular}{@{}lrrr@{}}
%     \toprule
%  & \Movielens & \Resolver & \TracePort \\ 
% \midrule
% \Saboteur & 40075	&	1282	&	9730 \\ 
% Unrestricted BTC &	39351	&	1122	&	8843 \\ 
%     \bottomrule 
% \end{tabular}
%  \caption{Training error, data size $85\%$}
%    \label{tab:overfitting:train}
%  \end{table}

%  \begin{table}[t]
%   \centering
%   \small
% \begin{tabular}{@{}lrrr@{}}
%     \toprule
%  & \Movielens & \Resolver & \TracePort \\ 
% \midrule			
% \Saboteur & 7258	&	225	&	1144\\ 
% Unrestricted BTC &	8256	&	274	&	1554 \\ 
%     \bottomrule 
% \end{tabular}
%  \caption{Testing error, data size $15\%$}
%    \label{tab:overfitting:test}
%  \end{table}

\begin{table}[tb]
 \topcaption{Training and testing errors on real-world data. The number of clusters was $k=15$ for \Movielens and \TracePort and $k=10$ for \Resolver.}
   \label{tab:real:overfitting}
  \centering
  \small
\begin{tabular}{@{}llrrr@{}}
    \toprule
&  & \Movielens & \Resolver & \TracePort \\ 
\midrule
Train & \SaboteurIU & 40075	&	1282	&	9730 \\ 
         & Unrestricted BTC &	39351	&	1122	&	8843 \\ 
\midrule
Test & \SaboteurIU & 7258	&	225	&	1144\\ 
        & Unrestricted BTC &	8256	&	274	&	1554 \\ 
\bottomrule 
\end{tabular}
 \end{table}

%%% Local Variables: 
%%% mode: latex
%%% TeX-master: "main"
%%% End: 

As expected, the Unrestricted BTC achieves lower training error, but larger testing error. This agrees with our synthetic experiments, and strongly indicates that using the rank-1 centroids helps to avoid overfitting.

\subsubsection{Selecting the Number of Clusters}
\label{sec:real:mdl}
We tested the use of MDL to select the number of clusters with two real-world data sets, \Enron and \Mag. To that end, we ran \Saboteur with every possible number of clusters (i.e.\ from having all slices in one cluster to having one cluster for each slice). Further, we computed the description length of representing the data with no clusters, i.e.\ having the full data explained in the $L(D\mid M)$ part. The description lengths for the different numbers of clusters are presented in Figure~\ref{fig:real:mdl}. For \Enron, the number of clusters that gave the smallest description length was $3$, while for \Mag it was $16$.

\begin{figure}[tb]
  \centering
  \subfigure[\Enron\label{fig:real:mdl:enron}]{\includegraphics[width=\subfigwidth]{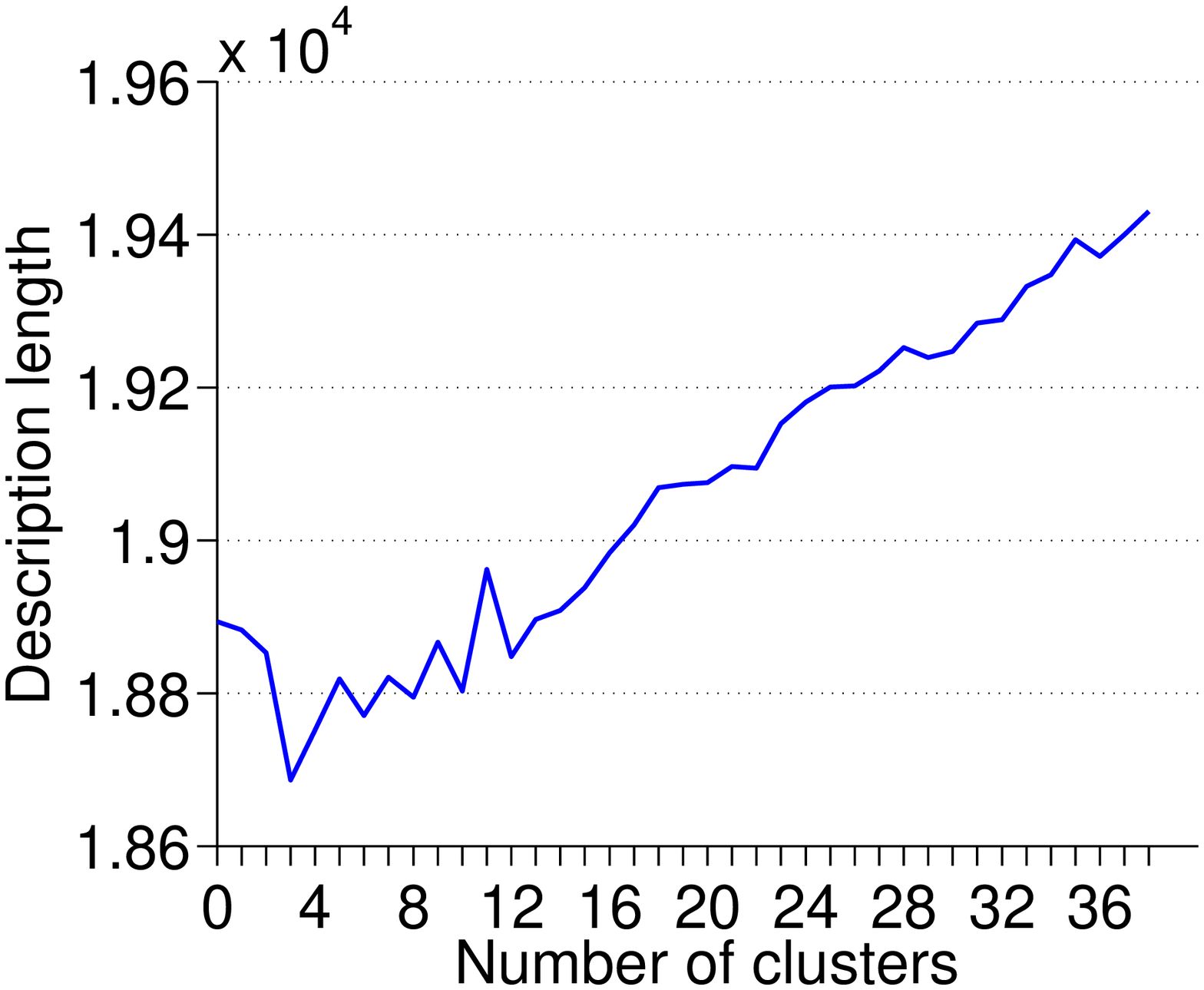}}
  \hspace{\subfigspace}
  \subfigure[\Mag\label{fig:real:mdl:mag}]{\includegraphics[width=\subfigwidth]{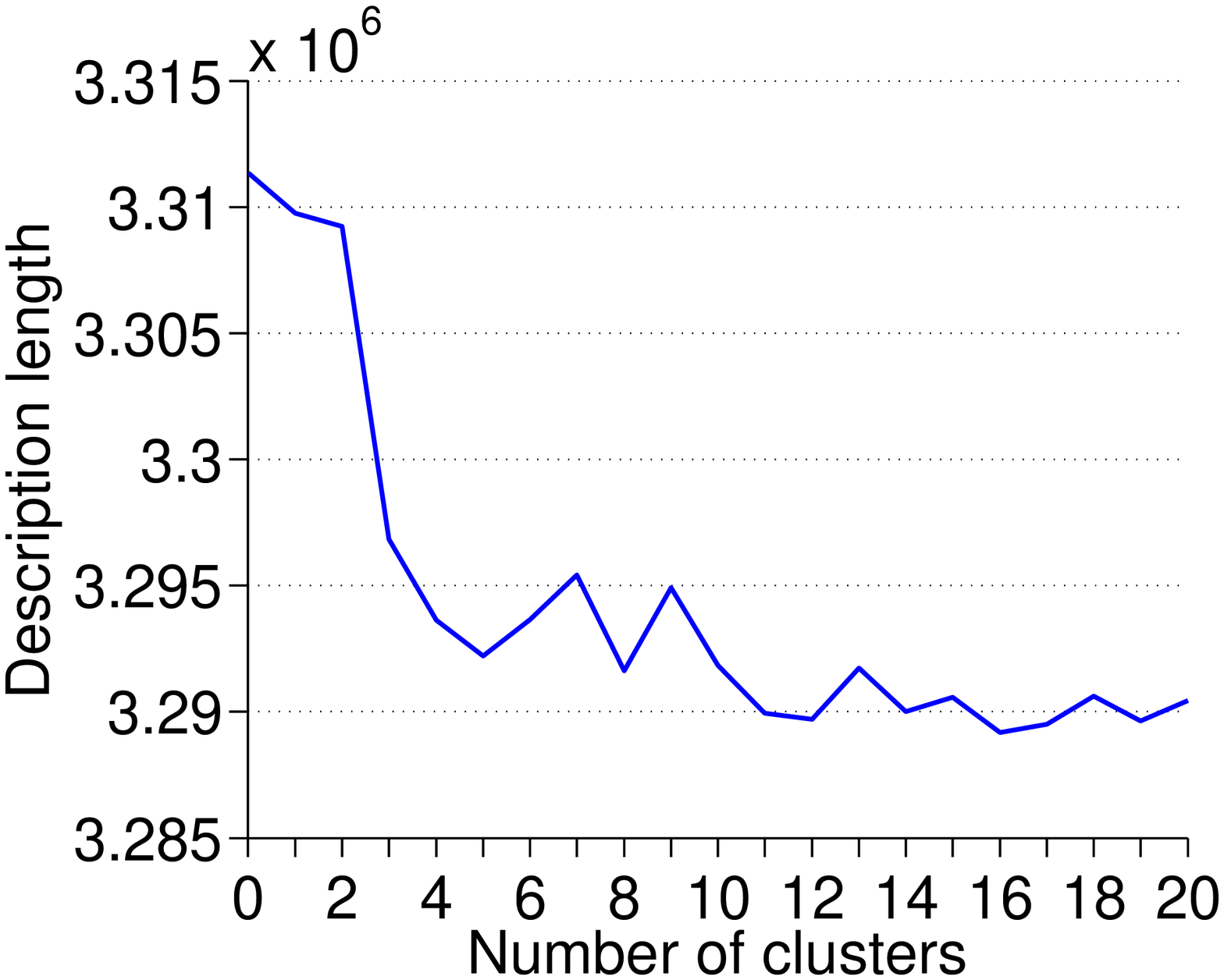}}
  \caption{Description length vs. number of clusters.}
  \label{fig:real:mdl}
\end{figure}

Both \Enron and \Mag show non-smooth behaviour in the description length. This is likely partially due to \Saboteur's randomized behaviour and partially due to the fact that \Saboteur does not consider the description length when it builds the clustering. Nevertheless, especially with \Enron (Figure~\ref{fig:real:mdl:enron}), the selection of the best number of clusters -- in the MDL's sense -- seems clear.

\subsubsection{Interpretability of the Results}
\label{sec:real:anecdotal}
For the final experiment with the real-world data, we studied some of the results we obtained from \Saboteur in order to see if they can provide insights to the data. We report results from three data sets: \Mag, \Delicious, and \Lastfm.  For \Mag, we can interpret all three modes (movies, actors, and genres); for the other two, we can only interpret two of the three modes (tags and URLs, and tags and artists, respectively), as the third mode corresponds to user IDs. 

For, \Mag, we used $16$ clusters (as suggested by MDL), and clustered the genres. As there are only $20$ genres in the data, this means many clusters contained only one genre; this, however, is not surprising, as we would expect the sets of movies that fall in different genres be generally rather different. The results picked famous movies of the genre as the centroids: for example, the cluster of \emph{animation} and \emph{comedy} genres had the 1995 animation \emph{Toy Story} appearing in its centroid. Another cluster, containing genres \emph{adventure} and \emph{fantasy}, had as its centroid  Peter Jackson's \emph{Lord of the Rings} trilogy, and its main cast. Similarly to \emph{Toy Story}, \emph{Lord of the Rings} movies are arguably stereotypical adventure--fantasy movies.

The centroids for the \Mag data contained very few movies (e.g.\ the three \emph{Lord of the Rings} movies). This is due to the fact that the data sets are very sparse, and the algorithm can only pick up the most prominent elements in the centroids. With even sparse data sets, such as \Delicious and \Lastfm, the centroids get even sparse, perhaps containing just a singleton element from some mode. This is a common problem to many Boolean methods \citep[see, e.g.][]{miettinen08discrete}. To alleviate it, we used a method proposed by \citet{miettinen08discrete}, and weighted the reconstruction error. Normally, representing a $0$ as a $1$ and representing a $1$ as a $0$ both increase the error by $1$, but for the next experiments, we set the weight so that representing a $1$ with a $0$ causes $10$ times more error as representing $0$ as a $1$. 

For \Delicious and \Lastfm, we mined $30$ clusters. Even with the $10$-fold weight, results for \Delicious were rather sparse (indicating that the users' behaviour was very heterogeneous). Some clusters were very informative, however, such as one with tags \emph{software}, \emph{apple}, \emph{mac}, and \emph{osx} that contained bookmarked pages on \emph{HFS for Windows}, \emph{iVPN}, \emph{Object-Oriented Programming with Objective-C}, and \emph{iBooks and ePub}, among others. For \Lastfm, the clusters were generally more dense, as users tend to tag artists more homogeneously. For example, the cluster containing the tags \emph{pop} and \emph{rnb} contained the superstar artists such as Beyonc\'e, Britney Spears, Katy Perry, and Miley Cyrus (and many more). But importantly, one could also find less-known artists: in a cluster containing tags \emph{dance}, \emph{trance}, \emph{house}, \emph{progressive trance}, \emph{techno}, and \emph{vocal trance}, we found artists such as \emph{Infected Mushroom}, \emph{T.M.Revolution}, \emph{Dance Nation}, and \emph{RuPaul}.

\subsection{Discussion}
\label{sec:exp:discusion}

All the experiments show that \Saboteur can hold its own even against methods that are theoretically superior: For reconstruction error (or similarity) with both synthetic (Sections~\ref{sec:varying-noise}--\ref{sec:varying-density}) and real-world data (Section~\ref{sec:real:accuracy}), \Saboteur (and \SaboteurIU) achieve results that are the best or close to the best results, notwithstanding that other methods, especially \CPAPRr and \ParCuber, benefit from significantly less constrained forms of decomposition, coupled with the post-hoc rounding. These relaxations also come with a price, as both methods create significantly denser factor matrices, and -- as can be seen in Sections~\ref{sec:scalability} and~\ref{sec:real:scale} -- are unable to scale to denser data. Compared to the Boolean CP factorization methods, \BCPALS and \walknmerge, \Saboteur is comparable (or sometimes clearly better) in reconstruction error and density of the factor matrices, while being significantly more scalable.

Our hypothesis that the rank-1 centroids help with overfitting was also confirmed (Sections~\ref{sec:generalization-test} and~\ref{sec:real:overfitting}): While the Unrestricted BTC obtained lower training errors (as expected), \Saboteur obtained better results with testing data in all real-world data sets and in most synthetic experiments (with \TracePort, \Saboteur's testing error was more than $25\%$ better than that of Unrestricted BTC's).  

Our experiments also show that MDL can be used to select the number of clusters automatically, even though we do think that more studies are needed to better understand the behaviour of the description length. Also, when studying the results, it seems obvious that \Saboteur can return results that are easy to interpret and insightful; in part, this is due to the Boolean algebra, and in part, due to the clustering.

The last open question is whether one should use \Saboteur or \SaboteurIU. Judging by the results, \SaboteurIU can provide tangible benefits over \Saboteur, however, with a significant price in the running time. Thence, we think that it is best to start with \Saboteur, and only if the user thinks she would benefit from more accurate results, move to \SaboteurIU.

%%% Local Variables: 
%%% mode: latex
%%% TeX-master: "main"
%%% End: 

\section{Related Work}
\label{sec:relatedWork}

Tensor clusterings have received some research interest in recent years. The existing work can be divided roughly into two separate approaches: on one hand \citet{jegelka09approximation} study the problem of clustering simultaneously all modes of a tensor (tensor co-clustering), and on the other hand, \citet{huang08simultaneous} and \citet{liu10hybrid} (among others) study the problem where only one mode is clustered and the remaining modes are represented using a low-rank approximation. The latter form is closer to what we study in this paper, but the techniques used for the continuous methods do not apply to the binary case. 

Normal tensor factorizations are well-studied, dating back to the late Twenties. The Tucker and CP decompositions were proposed in the Sixties \citep{tucker66some} and Seventies \citep{carroll70analysis,harshman70foundations}, respectively. The topic has nevertheless attained growing interest in recent years, both in numerical linear algebra and computer science communities. For a comprehensive study of recent work, see \citet{kolda09tensor}, and the recent work on scalable factorizations by \citet{papalexakis12parcube}.

The first algorithm for Boolean CP factorization was presented by \citet{leenen99indclas}, although without much analysis. Working in the framework of formal tri-concepts, \citet{belohlavek12optimal} studied the exact Boolean tensor decompositions, while \citet{ignatov11triconcepts} studied the problem of finding dense rank-1 subtensors from binary tensors. In data mining \citet{miettinen11boolean}, studied computational complexity and sparsity of the Boolean tensor decompositions. Recently \citet{erdos13walknmerge} proposed a scalable algorithm for Boolean CP and Tucker decompositions, and applied that algorithm for information extraction~\citep{erdos13discovering}. The tri-concepts are also related to closed $n$-ary relations, that is, $n$-ary extensions of closed itemsets \citep[see, e.g.][]{cerf09closed,cerf13closed}. For more on these methods and their relation to Boolean CP factorization, see \citet{miettinen11boolean}.

\citet{miettinen08discrete} presented the Discrete Basis Partitioning problem for clustering binary data and using binary centroids. They gave a $(10+\varepsilon)$ approximation algorithm based on the idea that any algorithm that can solve the so-called binary graph clustering (where centroids must be rows of the original data) within factor $f$, can solve the arbitrary binary centroid version within factor $2f$. Recently, \citet{jiang14pattern} gave a 2-approximation algorithm for the slightly more general case with $k+1$ clusters with one centroid restricted to be an all-zero vector. 

The maximization dual of binary clustering, the hypercube segmentation problem, was studied by \citet{kleinberg98microeconomic,kleinberg04segmentation} and they also gave three algorithms for the problem, obtaining an approximation ratio of $2(\sqrt{2} - 1)$. Later, \citet{seppanen05upper} proved that this ratio is tight for this type of algorithms and \citet{alon99two} presented a PTAS for the problem.

Recently, \citet{kim11approximate} proposed the tensor\-/relational model, where relational information is represented as tensors. They defined the relational algebra over the (decompositions of) the tensors, noting that in the tensor-relational model, the tensor decomposition becomes the costliest operation. Hence, their subsequent work~\citep{kim12decomposition,kim14pushing-down} has concentrated on faster decompositions in the tensor-relational framework.

%%% Local Variables: 
%%% mode: latex
%%% TeX-master: "main"
%%% End: 

\section{Conclusions and Future Work}
\label{sec:conclusions}

We have studied the problem of clustering one mode of a 3-way binary tensor while simultaneously reducing the dimensionality in the two other modes. This problem bears close resemblance to the Boolean CP tensor decomposition, but the additional clustering constraint makes the problem significantly different. The main source of computational complexity, the consideration of overlapping factors in the tensor decomposition, does not play a role in BCPC. This lets us design algorithms with provable approximation guarantees better than what is known for the Boolean matrix and tensor decompositions.

Our experiments show that the \Saboteur algorithm obtains comparable reconstruction errors compared to the much less restricted tensor factorization methods while obtaining sparse factor matrices and overall best scalability. We also showed that restricting the cluster centroids to rank-1 matrices significantly helped the generalization to unseen data.

The idea of restricting the clustering centroids to a specific structure could be applied to other types of data, as well. Whether that will help with the curse of dimensionality remains an interesting topic for the future research, but based on the results we presented in this paper, we think this approach could have potential. On the other hand, one could also consider higher-rank centroids, moving from \BCPC\ towards Unrestricted BTC. In a sense, the rank of the centroid can be seen as a regularization parameter, but the viability of this approach requires more research, both in Boolean and in continuous domains.

We also argue that the similarity -- as opposed to distance or error -- is often more meaningful metric for studying the (theoretical) performance of data mining algorithms due to its robustness towards small errors and more pronounced effects of large errors. To assess whether similarity is better than dissimilarity, more methods, existing and novel, should be analysed, and the results should be contrasted to those obtained when studying the dissimilarity and real-world performance. Our hypothesis is that good theoretical performance in terms of similarity correlates better with good real-world performance than good (or bad) performance correlates with the error. 

Finally, our algorithms could be used to provide fast decompositions of Boolean tensors in the tensor-relational model.

%The essential piece, the maximum-similarity binary rank-1 approximation achieves an approximation ratio of $0.828$ in  $O(nm\min\{n,m\})$ time, and with that dominates the running time. Faster algorithms for the rank-1 approximation (with better approximation guarantees) would have an instant impact on \Saboteur.

%%% Local Variables: 
%%% mode: latex
%%% TeX-master: "main"
%%% End: 

%% spbasic is the style for natbib citations; spmpsci is the style for numbered citations
\bibliographystyle{spbasic}
\bibliography{library}  
\end{document}